\documentclass[10pt, conference, letterpaper]{IEEEtran} 

\usepackage{geometry}
\usepackage{amsmath}
\usepackage{amssymb}
\usepackage{dsfont}
\usepackage{enumerate}
\usepackage{amsthm}
\usepackage{graphicx}
\usepackage{algorithm}% http://ctan.org/pkg/algorithms
\usepackage{algpseudocode}% http://ctan.org/pkg/algorithmicx
\usepackage{tikz}
\usepackage{booktabs}
\usepackage{bbold}
\usetikzlibrary{automata, positioning, arrows}
\usepackage[utf8]{inputenc}
\usepackage[english]{babel}
\usepackage{mathtools}

\theoremstyle{theorem}
\newtheorem{theorem}{Theorem}
\usepackage{cite}
\theoremstyle{theorem}
\newtheorem{remark}{Remark}

\newtheorem{assumption}{Assumption}
\newtheorem{corollary}[theorem]{Corollary}

\theoremstyle{definition}
\newtheorem{example}{Example}

\theoremstyle{theorem}
\newtheorem{problem}{Problem}
\newtheorem{lemma}[theorem]{Lemma}

\theoremstyle{definition}
\newtheorem{definition}{Definition}

\theoremstyle{proof}

\usepackage{booktabs}
\theoremstyle{proposition}
\newtheorem{proposition}[theorem]{Proposition}

\tikzset{
	->, % makes the edges directed
	>=stealth', % makes the arrow heads bold
	initial text=$ $, % sets the text that appears on the start arrow
}

\definecolor{purple}{rgb}{1,0.5,1}
\definecolor{green}{rgb}{0.5,1,0.5}
\definecolor{red}{rgb}{1,0.5,0.5}

\usepackage{subcaption}
\usepackage{url}
 \IEEEoverridecommandlockouts 
\graphicspath{{images/}}

\title{ \vspace{0.25in} Least Inferable Policies for Markov Decision Processes}
\author{Mustafa O. Karabag\thanks{ M. O. Karabag is with the Department of Electrical and Computer Engineering, University of Texas at Austin. e-mail: karabag@utexas.edu }, Melkior Ornik\thanks{M. Ornik is with the Institute for Computational Engineering and Sciences, University of Texas at Austin. e-mail: mornik@ices.utexas.edu}, Ufuk Topcu\thanks{U. Topcu is with the Department of Aerospace Engineering and Engineering Mechanics and the Institute for Computational Engineering and Sciences, University of Texas at Austin. e-mail: utopcu@utexas.edu} }
\newgeometry{margin=0.75in, bottom=0.75in, top=0.75in}
\begin{document}
	\onecolumn
	\maketitle
	
	\begin{abstract}
		In a variety of applications, an agent's success depends on the knowledge that an adversarial observer has or can gather about the agent's decisions. It is therefore desirable for the agent to achieve a task while reducing the ability of an observer to infer the agent's policy. We consider the task of the agent as a reachability problem in a Markov decision process and study the synthesis of policies that minimize the observer's ability to infer the transition probabilities of the agent between the states of the Markov decision process. We introduce a metric that is based on the Fisher information as a proxy for the information leaked to the observer and using this metric formulate a problem that minimizes expected total information subject to the reachability constraint. We proceed to solve the problem using convex optimization methods. To verify the proposed method, we analyze the relationship between the expected total information and the estimation error of the observer, and show that, for a particular class of Markov decision processes, these two values are inversely proportional.
	\end{abstract}
	
	\theoremstyle{definition}
	
	\section{Introduction}
	%setting block
	We consider a setting in which an agent is supposed to accomplish a task in a stochastic environment while an observer that is potentially adversarial tries to infer the characteristics of the agent's behavior. For a scenario where predictable behaviors may put the success of the task at risk, it is crucial for an agent to conceal its strategy. In this paper, we study the synthesis of policies that enable an agent to achieve its task while limiting the ability of the observer to infer.
	%observation block
	
	We model the behavior of the agent by a Markov decision process (MDP). The agent follows a policy to achieve its objective, for example, reaching a set of target states with high probability. This policy determines the transition probabilities of the agent between the states of the MDP. The observer can observe the transitions of the agent at a subset of the states and, solely based on the observed transitions, infers the agent's transition probabilities at these states. As a counter objective, the agent aims to choose its policy such that it limits the ability of the observer to infer the transition probabilities in addition to achieving the agent's task with high probability. 
	
	A policy can limit the ability of the observer to infer by minimizing the amount of information on the transition probabilities that the observer can gather from each observed transition. We introduce a metric, \textit{transition information}, to measure the amount of information that a single transition leaks to the observer. This metric is related to the Fisher information which measures the amount of information that a random variable has on a parameter \cite{frieden2004science}. An observer that is trying to estimate the parameter would have high expected estimation error if the random variable has low Fisher information on the parameter. The notion of transition information generalizes the Fisher information by providing a scalar value describing the information leaked for the agent's transition that is parametrized by the transition probabilities. 
	
	While the notion of transition information is appropriate for a single observed transition, we also need to consider the effect of the number of observed transitions on the ability of the observer to infer. A policy that solely minimizes the transition information for each observed state adjusts the transition probabilities to the successor states as close to each other as possible since the uniform distribution of the successor states minimizes the transition information for a state. However, this approach might increase the number of observed transitions, and the observer may be able to infer the transition probabilities due to high number of observed transitions. Hence a policy that minimizes the ability of the observer to infer the transition probabilities must also take into account the number of observed transitions and balance the number of observed transitions and the transition information of each observed transition.
	
	We account for the two quantities of interest, the number of observed transitions and the transition information of each observed transition, through a unified notion of \textit{expected total information} --- the expected sum of transition informations over a path generated by the agent's policy. We propose to compute a policy that has the minimum expected total information subject to the constraint that the task of the agent is completed with high probability. 
	
	To the best of our knowledge, the proposed method is the first policy synthesis method that uses the Fisher information for planning in MDPs against an adversary.  The method introduced in \cite{alpcan2015information} uses the Fisher information for learning and control in unknown systems that are modeled by MDPs. However, in contrast to our approach, \cite{alpcan2015information} aims to increase the information gathered from transitions. A-optimality criterion \cite{emery1998optimal} for experiment design aims to minimize the total variance of estimators by minimizing the trace of the inverse Fisher information matrix. The transition information is the reciprocal of the trace of inverse Fisher information matrix. In contrast, by minimizing the transition information, we aim to maximize the total variance of estimators unlike A-optimality criterion.	In terms of the use of Fisher information, the closest works to the method proposed in this paper are \cite{farokhi2017optimal} and \cite{farokhi2017fisher}. The methods introduced in \cite{farokhi2017optimal} and \cite{farokhi2017fisher} use the Fisher information to preserve privacy for database systems and smart meters, respectively, and they do not deal with MDPs. Planning in stochastic control settings in the presence of an adversary has been substantially explored previously; the works closest to our paper are \cite{agmon2008multi, paruchuri2006security, savas2018entropy}. The reference \cite{agmon2008multi} provides a method for multi-agent perimeter patrolling scenarios and is not applicable to MDPs in general. Papers \cite{paruchuri2006security, savas2018entropy} propose to randomize the policy of an agent by maximizing the entropy of an induced stochastic process. While, for an MDP, increasing the entropy of a process increases randomness of the paths, it does not necessarily limit the ability of an observer to infer the transition probabilities. 
	
	The rest of the paper is organized as follows. Section \ref{section:preliminaries} provides necessary background on the proposed method. In Section \ref{section:probstatement}, the definition of information and the problem formulation are presented. Section \ref{section:methodology} includes the methodology to synthesize the policy that has minimum expected total information subject to a reachability constraint by convex optimization problems. In Section \ref{section:errorbound}, we show the relationship between considered problems and estimation errors of the observer. We present numerical examples in Section \ref{section:examples} and conclude with suggestions for the future work in Section \ref{section:conclusion}. We discuss some special cases of the proposed method in Appendix \ref{appendix:unobservedmec} and give the proofs for the technical results of this paper in Appendix \ref{appendix:proofs}.

	\section{Preliminaries} \label{section:preliminaries}
	In this section, we present some of the concepts and notation used in the rest of the paper.
	
	We use $[n]$ for the set $\lbrace{ 1,\ldots, n \rbrace}$. For a finite set $D$, we denote the power set with $2^{D}$ and cardinality with $|D|$. $\mathbb{E}[\Theta]$ denotes the expectation of the random variable $\Theta$ and $\text{Var}(\Theta)$ denotes the variance of $\Theta$ which is $\mathbb{E}\left[ (\Theta - \mathbb{E}[\Theta])(\Theta - \mathbb{E}[\Theta])^T \right]$. We use $\mathds{1}_{D}$ for the indicator function of a set $D$ where $\mathds{1}_{D}(x)=1$ if $x \in D$ and  $\mathds{1}_{D}(x)=0$ otherwise.
	\subsection{Markov Decision Processes}
	A \textit{Markov decision process} is a tuple $\mathcal{M} = (S, \mathcal{A}, \mathcal{P}, s_0)$ where $S$ is a finite set of states, $\mathcal{A}$ is a finite set of actions, $\mathcal{P}: S \times \mathcal{A} \times S \to [0,1]$ is the transition probability function, and $s_0$ is the initial state. We denote $\mathcal{P}(s,a,q)$ by $\mathcal{P}_{s,a,q}$. $\mathcal{A}(s)$ denotes the set of available actions at state $s$ where $\sum_{q \in S} \mathcal{P}_{s,a,q} = 1$ for all $a \in \mathcal{A}(s)$. We denote the successor states of state $s$ by $Succ(s)$ such that a state $q \in S$ if and only if there exists an action $a$ such that $\mathcal{P}_{s,a,q} > 0 $.  A state $s$ is \textit{absorbing} if it has only a single successor state that is itself, i.e.,  $\mathcal{P}_{s,a,s}=1$ for all $a \in \mathcal{A}(s)$.   
	
	A \textit{sub-MDP} $(C,D)$ of $\mathcal{M}$ is a pair where $C \subseteq S$ is non-empty and $D: C \to 2^{\mathcal{A}}$ is a function such that $a \ in D(s)$ only if $\mathcal{P}_{s,a,q} = 0$ for all $q \not \in C$. An \textit{end component} is a sub-MDP $(C,D)$ of $\mathcal{M}$ such that the digraph induced by $(C,D)$ is strongly connected. An end component $(C,D)$ is \textit{closed} if, for all $s \in C$, $Succ(s) \setminus C = \emptyset$. A \textit{maximal end component} $(C,D)$ is an end component where there is no end component $(C', D')$ such that $(C, D) \neq (C', D')$, $C \subseteq C'$, and $D \subseteq D'$.
	
	A \textit{policy} is a sequence $\pi = [\mu_0, \mu_1, \ldots]$ where each $\mu_t:S \times \mathcal{A} \to [0,1]$ is a function such that $\sum_{a \in \mathcal{A}(s)} \mu_t(s,a) = 1$ for every $s \in S$. A \textit{stationary policy} $\pi$ is a  a sequence $\pi = [\mu, \mu, \ldots]$ where $\mu:S \times \mathcal{A} \to [0,1]$ is a function such that $\sum_{a \in \mathcal{A}(s)} \mu(s,a) = 1$ for every $s \in S$. We denote the set of all policies by $\Pi(\mathcal{M})$ and the set of stationary policies by $\Pi^{St}(\mathcal{M})$. For a stationary policy $\pi$, we denote $\mu(s,a)$ by $\pi_{s,a}$. A stationary policy $\pi$ induces \textit{a Markov chain} $\mathcal{M}^{\pi} = (S, \mathcal{P}^{\pi},  s_0)$  from $\mathcal{M}$ where $\mathcal{P}^{\pi}:S \times S \to [0,1]$ is the transition probability function such that $$ \mathcal{P}^{\pi}(s,q) := \sum_{a \in \mathcal{A}(s)} \pi_{s,a} \ \mathcal{P}_{s,a,q}$$ for all $s,q \in S$. We denote $\mathcal{P}^{\pi}(s,q)$ by $\mathcal{P}^{\pi}_{s,q}$.

	A \textit{path} $\xi = s_0 s_1 s_2 \ldots$ is an infinite sequence of states under policy $\pi= [\mu_0, \mu_1, \ldots]$ such that $\sum_{a \in \mathcal{A}(s_t)} \mathcal{P}_{s_t,a,s_{t+1}} \mu_t(s_t,a) > 0$ for all $t\geq0$. The set of paths for $\mathcal{M}$ under policy $\pi$ is denoted by $Paths(\mathcal{M}^{\pi})$.
	
	% A $k$-length path fragment $\xi_k = s_0 s_1 s_2 \ldots s_{k-1}$ is a sequence of $k$-states such that $\sum_{a \in s_t} \mathcal{P}_{s_t,a,s_{t+1}} \mu_t(s_t,a) > 0$ for all $k > t \geq 0$. The set of $k$-length path fragments of is denoted by $Paths_{k}(\mathcal{M}^{\pi})$. 
	
	The reachability probability to the set $B$ of states, i.e., the probability of reaching a state $b \in B$ under policy $\pi$, is denoted by $\Pr^{\pi} (Reach [B])$. 
	
	The \textit{expected state residence time} at state $s$ is defined by $$x^{\pi}_{s} := \sum_{t=0}^{\infty} \Pr (s_t = s | s_0),$$ where $s_t$ is the state at time $t$. The expected state residence time $x^{\pi}_{s}$ is also equal to $ \mathbb{E}[N_{s,\xi}]$ where $N_{s,\xi}$ is the number of appearances of $s$ in the random path $\xi$ that is generated by the policy $\pi$. The \textit{expected state-action residence time} at state $s$ and action $a$ is defined by $$x^{\pi}_{s,a} := \sum_{t=0}^{\infty} \Pr (s_t = s | s_0) \mu_t(s_t,a).$$ The expected state-action residence time of a state and an action is the expected number of times that the action is taken at the state. For a stationary policy $\pi \in \Pi^{St}(\mathcal{M})$, $x^{\pi}_{s,a} = \pi_{s,a} \ x^{\pi}_{s}$.
	
	%		The expected state residence time of a state is the expected number of visits to the state. We can define the expected residence time using the number of visits. For the reachable, transient state $s$, the number of visits to the state $s$ is a random variable $N_s$ that is $0$ with probability $(1 - \Pr^{\pi}(Reach[s]|s_0))$ and $k>0$ with probability $\Pr^{\pi}(Reach[s]|s_0)\Pr^{\pi}(Reach[s]|s)^{k-1}(1-\Pr^{\pi}(Reach[s]|s))$. In words, the number of visits to a state is a mixed random variable that consists of a constant random variable $0$ and a geometric random variable. The expected state residence time can also be expressed as $$x^{\pi}_{s} = \mathbb{E}[N_s] = \Pr^{\pi}(Reach[s]|s_0) \dfrac{1}{1 - \Pr^{\pi}(Reach[s]|s)}.$$  

	\subsection{The Fisher Information and the Cram\'er-Rao Bound}
	Let the random variable $X$ represent the observed data from a random variable that is parametrized by $\Theta \in \mathbb{R}^n$. An \textit{estimator} is a function  $\hat{\Theta}: X \to \mathbb{R}^n$ that estimates $\Theta$ based on observed data. The estimator $\hat{\Theta}$ is an \textit{unbiased estimator} of $\Theta$ if $\mathbb{E}[\hat{\Theta}] - \Theta = 0.$ 
	
	The \textit{precision} of a random variable is the reciprocal of the variance of the random variable. For an unbiased estimator, its precision is the reciprocal of the mean squared error (MSE) of the estimator.
	
	The \textit{Fisher information} \cite{lehmann2006theory} $I_X(\theta)$ of a discrete random variable $X$ parametrized by $\theta \in \mathbb{R}$ is $$I_X(\theta) := - \sum_{x \in Supp(X)} \frac{\partial^2 \log(\Pr(X=x |\theta))}{\partial \theta^2} \Pr(X=x |\theta).$$ An important property of the Fisher information is additivity, that is, when the samples are drawn from i.i.d. random variables, the Fisher information based on $n$ samples $I_{X^n}(\theta)$ satisfies $I_{X^n}(\theta) = n I_{X}(\theta)$ where $I_{X}(\theta)$  is the Fisher information of one sample.
	
	The \textit{Cram\'er-Rao inequality} \cite{lehmann2006theory} defines a relationship between the variance of an unbiased estimator of parameter $\theta$ and the Fisher information on the parameter $\theta$. The inequality is stated as 
	\begin{equation} \label{cramerraoinequality}
	\text{Var}(\hat{\theta}) \geq I_X(\theta)^{-1}
	\end{equation} 		
	where $\hat{\theta}$ is any unbiased estimator of $\theta$.
	
	An unbiased estimator is \textit{efficient} if it achieves the Cram\'er-Rao bound.

	\section{Problem Statement} \label{section:probstatement}
	\label{information}
	
	Consider an agent whose behavior is governed by a Markov decision process (MDP) $\mathcal{M} = (S, \mathcal{A}, \mathcal{P}, s_0)$ where a stationary policy followed by the agent $\pi$ implemented on this MDP induces a Markov chain. An adversary which we call \textit{observer} observes the transitions and tries to infer the transition probabilities for a set $W$ of states in the induced Markov chain. We assume that the observer can only observe the transitions at the states in $W$ which we call \textit{observed states}, and has no side information.
	
	The problem we study is the synthesis of a policy for the agent with two  objectives: (i) reach a set $C_{reach}$ of states with probability higher than a given threshold $0 \leq \nu_{reach} \leq 1$ and (ii) minimize the amount of information leaked to the observer. 
	
	For the first objective, we assume that the transitions of the agent after reaching $C_{reach}$ are irrelevant, i.e., every state in $C_{reach}$ is absorbing and is not observed.
	
	For the second objective, we define the notion of \textit{transition information} to measure the amount of information leaked to the observer due to a transition.   
	\begin{definition}
		The \textit{transition information} of a state $s$ is defined by
		\begin{equation} \label{def:infogeneral}
		\iota^{\pi}_s := \frac{1}{\sum_{q \in Succ(s)} I_{Q}(\mathcal{P}^{\pi}_{s,q})^{-1}}
		\end{equation}
		where $Q$ is the random variable that is the successor state of state $s$. 
	\end{definition}
	
	We remark that the Fisher information and the transition information are analogous:
	\begin{itemize}
		\item The reciprocal of the Fisher information is a lower bound on the variance of an unbiased estimator for a single parameter.
		\item The reciprocal of the transition information is a lower bound on the variance of an unbiased estimator for many parameters.
	\end{itemize}
	For a state $s$, consider an unbiased estimator $\hat{\mathcal{P}_s}$ of transition probabilities. The reciprocal of the transition information $\iota^{\pi}_{s}$ is a lower bound on the variance of $\hat{\mathcal{P}_s}$: $$\text{Var}(\hat{\mathcal{P}_s}) \geq \frac{1}{\iota^{\pi}_s}.$$ 
	
	We use the transition information to define the \textit{total information} of a path. The total information of a path $\xi=s_0 s_1 s_2 \ldots$ is defined as the sum of each observed transition's transition information such that $$\iota^{\pi}_{W,\xi} := \sum_{t=0}^{\infty} \mathds{1}_{W}(s_t) \iota^{\pi}_{s_t}.$$ We then state the synthesis problem formally as follows:
	\theoremstyle{problem}
	\begin{problem}[Synthesis of Minimum-Information Admissible Policies] \label{problem}
		Given an MDP $\mathcal{M}=(S,\mathcal{A},\mathcal{P},s_0)$, a set $C_{reach}$ of states, a probability threshold $\nu_{reach}$, and the set $W$ of observed states, compute
		\begin{subequations}
			\label{problem:mininfleakage}
			\begin{align}
			&\underset{ \pi \in \Pi^{St}(\mathcal{M})}{\inf}
			& &   \mathbb{E} [ \iota^{\pi}_{W,\xi} ] \label{eq:pathprobmininf},
			\\
			& \text{s. t.}
			& & {\Pr} ^{\pi}(Reach[C_{reach}] ) \geq \nu_{reach} \label{cons:reachreq}
			\end{align}
		\end{subequations}
		where $\xi$ is a random path generated under policy $\pi$. If the optimal value is attainable, compute the optimal policy $\pi^*$.
	\end{problem}
	Hereafter we call the policies that satisfy the reachability constraint \textit{admissible policies} and an optimal policy for Problem \ref{problem} a \textit{minimum-information admissible policy}.

	%\begin{definition}
	%	The information-leakage of path $\xsi = s_0 s_1 s_2 \ldots$ is to one transition from state $s$ is $$\iota^{\pi}(s)= \left( \sum_{i=1}^{|Succ(s)|} \text{Var} (t_i)  \right)^{-1}$$ \textbf{what is variance of t} where $|Succ(s)| \neq 1$. If state $s$ has only one successor state we assume $\iota^{\pi}(s) = 0$. $\iota^{\pi}(s)$.
	%\end{definition}

	\begin{example} \label{example:succstate}
		We explain the characteristics of a minimum-information admissible policy through the MDP given in Figure \ref{fig:example1} with $\nu_{reach}=0$ for simplicity.
		
		\begin{figure}[h] 
			\centering
			\begin{tikzpicture}
			\node[state, initial]  (s0) {$s_0^o$};
			\node[state] [above right=of s0] (s1) {$s_1^o$};
			\node[state] [right=of s1] (s2) {$s_2$};
			\node[state] [below right=of s1] (s3) {$s_3$};
			\draw 
			(s0) edge node[above, sloped] {$\alpha,1$} (s1)
			(s0) edge node[below, sloped] {$\beta,1$} (s3)
			(s1) edge node[above, sloped] {$\alpha,1$} (s2)
			(s1) edge node[above, sloped] {$\beta,1$} (s3)
			(s2) edge[loop right] node[right] {$\alpha,1$} (s2)
			(s3) edge[loop right] node[right] {$\alpha,1$} (s3);
			\end{tikzpicture}
			\caption{An MDP with 4 states. A label $a,p$ of a transition refers to the transition that happens with probability $p$ when action $a$ is taken. The states marked with the superscript $o$ are observed.}
			\label{fig:example1}
		\end{figure}
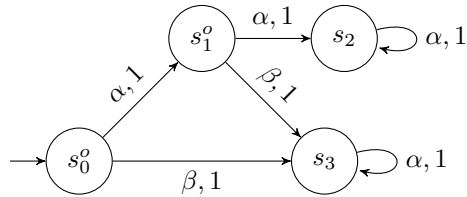
		
		The goal of the agent is to find a policy that minimizes the expected total information. Consider the policy at state $s_1$ and note that the policy decision at $s_1$ does not affect information leaked from state $s_0$ since it does not change the expected residence time at $s_0$. Hence we may only consider the information leaked from $s_1$. If the agent chooses a deterministic policy, the observer can estimate the transition probabilities with no error even after observing a single transition, which means infinite leaked information. Therefore, it is expected that the agent randomizes the transition probabilities. Formally, we explain the reasoning by the fact that the Fisher information is minimized for a $Ber(p)$ random variable with $p=0.5$. Similarly at state $s_0$, the agent randomizes the transition probabilities. However, unlike $s_1$, the policy at $s_0$ affects the information leaked from $s_1$. As the agent decreases the probability of taking action $\alpha$ at state $s_0$, the expected number of visits to state $s_1$ decreases and consequently information leaked from $s_1$ decreases. Hence, the agent must take the action $\beta$ with a greater probability than the action $\alpha$. On the other hand, taking the action $\beta$ with high probability increases the information leaked from $s_0$. We expect that, under this trade-off, the agent must choose a policy that takes both actions, but the action $\beta$ more likely. Numerically, the optimal policy is $\pi_{s_0, \alpha} = 0.38$, $\pi_{s_0, \beta} = 0.62$, $\pi_{s_1, \alpha} = 0.5$, and $\pi_{s_1,\beta} = 0.5$. 
	\end{example}

	\begin{remark}
		Note that if the transition probabilities are not constant and change between observations, measurement of inference with a transition information is not meaningful since we assume underlying probability distribution is constant. To have a well-defined problem, we only focus on agents that have to follow stationary policies and we search the optimal policies only in the stationary policies.
	\end{remark}
	
	\section{Synthesis of Minimum-Information Admissible Policies}  \label{section:methodology}
	
	For an MDP $\mathcal{M}=(S, \mathcal{A}, \mathcal{P},s_0)$, we aim to find a minimum-information admissible policy $\pi$ that minimizes the expected total information of a path subject to the reachability constraint $\Pr^{\pi}(Reach[C_{reach}]) \geq \nu_{reach}$ where the set of observed states is $W$. In this section, we represent the expected total transition transition information in terms of expected state-action residence times, show the existence of a minimum-information admissible policy, and give an optimization problem whose solution is a minimum-information admissible policy. We also show that the proposed optimization problem is convex in the expected state-action residence time parameters and hence can be solved using off-the-shelf convex optimization tools. 
	
	Note that the Fisher information for a parameter is well-defined if the regularity conditions are satisfied. These conditions require that the distributions depending on the parameter have a common support that is independent of the parameter \cite{lehmann2006theory}. For a random variable $P \sim Ber(p)$, the Fisher information is not defined when $p = 0$ or $p = 1$ since the probability distribution of $P$ does not have a common support. However, such a case practically corresponds to infinite Fisher information, which means that the value of the parameter is estimated exactly even after a single observation. We assume that the Cram\'er-Rao lower bound is zero if the Fisher information is infinite. 
	
	Consider a state $w \in W$ whose successor state is denoted by the random variable $Q$. For each $q \in Succ(w)$, we have  $$ I_{Q}(\mathcal{P}^{\pi}_{w,q}) =I_{\mathds{1}_{q}(Q)}(\mathcal{P}^{\pi}_{w,q})=\frac{1}{\mathcal{P}^{\pi}_{w,q}(1-\mathcal{P}^{\pi}_{w,q})}$$ where $\mathds{1}_{q}(Q)$ is a $Ber(\mathcal{P}^{\pi}_{w,q})$ random variable. The transition information of a state $w$ is a function
	\begin{equation*}
		\iota_w: \lbrace \mathcal{P}_w \in \mathbb{R}^{|Succ(w)|}:
		\sum_{q \in Succ(w)}\mathcal{P}_{w,q} =1, \mathcal{P}_{w,q} \geq 0 \rbrace \to \mathbb{R} \cup \lbrace{ \infty \rbrace}
	\end{equation*} and, under policy $\pi$, is equal to
	\begin{subequations}
		\begin{align}
		\iota^{\pi}_{w} = \left( \sum_{q \in Succ(w)} \mathcal{P}^{\pi}_{w,q} (1- \mathcal{P}^{\pi}_{w,q}) \right)^{-1}. \label{def:info}
		\end{align} 
	\end{subequations}
	
	\begin{remark}
		The categorical random variable $Q$ has the distribution $\mathcal{P}^{\pi}_{w,q}$ where $q \in Succ(w)$. The covariance matrix $\Sigma$ of $Q$ has diagonal entries $\mathcal{P}^{\pi}_{w,q}(1-\mathcal{P}^{\pi}_{w,q})$. The transition information of state $w$ given in \eqref{def:info} is also equal to $ \normalfont \text{tr}(\Sigma)^{-1}$. Since $Q$ is a categorical random variable, a sample mean estimator achieves the Cram\'er-Rao bound for a single transition. However, since the observed data consists of transitions from a path and the transitions are not independent in general, a sample mean estimator is not necessarily unbiased and efficient.  
	\end{remark}
	
	We now construct the optimization problem whose solution gives the expected state-action residence times for a minimum-information admissible policy. First, we rewrite \eqref{def:info} as
		\begin{equation}
	\iota^{\pi}_{w} = \left ( \sum_{q \in Succ(w)} \left (  \sum_{a \in \mathcal{A}(w)} \frac{ x^{\pi}_{w,a}}{\sum_{a' \in \mathcal{A}(w)}x^{\pi}_{w,a}} \mathcal{P}_{w,a,q} \right ) \left ( 1-\sum_{a \in \mathcal{A}(w)} \frac{ x^{\pi}_{w,a'}}{\sum_{a' \in \mathcal{A}(w)}x^{\pi}_{w,a'}}  \mathcal{P}_{w,a,q} \right ) \right )^{-1} \label{eq:restimeprob}
		\end{equation}
	using the definitions of the induced Markov chain and expected state-action residence times.

	We assume that the optimal value of Problem \ref{problem} is finite. If the optimal value is infinite any admissible policy is a minimum-information admissible policy.
	
	\begin{proposition} \label{proposition:replace}
		For an MDP $\mathcal{M}$, if $\mathbb{E} [ \iota^{\pi}_{W,\xi} ]$ is finite where $\xi$ is a path generated randomly under a policy $\pi \in \Pi^{St}(\mathcal{M})$, then $$\mathbb{E} [ \iota^{\pi}_{W,\xi} ] = \sum_{w \in W}x^{\pi}_{w} \iota^{\pi}_{w}.$$
	\end{proposition}
	
%		\begin{figure}[h] 
%		\centering
%		\includegraphics[width=0.3\paperwidth]{fisherinfo3}
%		\caption{The expected total information $x^{\pi}_{w}\iota^{\pi}_{w}$ for state $w$ that has two successor states $q_1$ and $q_2$. }
%		\label{fig:leakageofastate}
%	\end{figure}
%	
%	Figure \ref{fig:leakageofastate} shows the dependency between the expected total information $x^{\pi}_w \iota^{\pi}_w$ of a state $w$ that has two successor states, $q_1$ and $q_2$, and the expected state-state residence times, $x^{\pi}_{w, q_1}$ and $x^{\pi}_{w, q_2}$. For a fixed value of the expected state residence time $x^{\pi}_w$, the expected total information  $x^{\pi}_w \iota^{\pi}_w$ is minimized when the transition probabilities to the successor states are equal. For fixed transition probabilities, the expected total information $x^{\pi}_w \iota^{\pi}_w$ linearly increases with the expected state residence time. 
	
	Note that the expected total information $x^{\pi}_{w} \iota^{\pi}_{w}$ of a state $w$ has some undefined points on the domain $x^{\pi}_{w,a}\geq 0 $ where $a \in \mathcal{A}(w)$. We define the function at such points as follows:  
	\begin{itemize} \label{assumption:infleakage}
		
		\item  If the expected state residence time is zero, i.e.,  $x^{\pi}_{w} = \sum_{a \in \mathcal{A}(w)} x^{\pi}_{w,a} = 0$, then $x^{\pi}_{w} \iota^{\pi}_{w} := 0$. Since the state will never be visited, the observer cannot get information on the transition probabilities.
		
		\item  If $w$ deterministically transitions to one of the successor states and expected residence time is greater than zero, i.e., there exists a state $q\in Succ(w)$ such that $\sum_{a \in \mathcal{A}(w)} x^{\pi}_{w,a} \mathcal{P}_{w,a,q} > 0$ and $\sum_{a \in \mathcal{A}(w)} x^{\pi}_{w,a} \mathcal{P}_{w,a,q'} = 0$ for all $q' \in Succ(w) \setminus q$, then $x^{\pi}_{w} \iota^{\pi}_{w} :=\infty$. Since the observer can estimate the transition probabilities even after a single observation and there is a positive probability that the state will be visited, the expected total information is infinite.
		
		\item If the expected state residence time at $w$ is infinite, i.e., $x^{\pi}_{w} = \sum_{a \in \mathcal{A}(w)} x^{\pi}_{w,a} = \infty$, then $x^{\pi}_{w} \iota^{\pi}_{w} :=\infty$. Since the observed distribution of transitions converges to the transition probabilities,  the expected total information is infinite.
	\end{itemize}
	
	We represent the stationary policies of the agent with a set of constraints which use the expected state-action residence times. A stationary policy makes each state either recurrent or transient. We need to identify the states that can be reachable and recurrent. If a policy leaks finite information, a set of states can be reachable and recurrent if and only if they belong to an end component and are not observed since the recurrence of a reachable observed state results in infinite expected total information. 
	
	\begin{definition}
		An \textit{unobserved end component} (UEC) is a sub-MDP $(C,D)$ such that the digraph induced by $(C,D)$ is strongly connected and $C \cap W = \emptyset$.	An \textit{unobserved maximal end component} (UMEC) $(C,D)$ is a UEC where  $C \subset S$  and there is no UEC $(C', D')$ such that $(C, D) \neq (C', D')$, $C \subseteq C'$, and $D \subseteq D'$.	
	\end{definition}
	
	We denote the set of states that belong to some UMEC by  $C_{end}$. After reaching $C_{end}$, the agent can follow a stationary policy that always stays in the UMEC and leaks no more information. For example, $s_2$ is a UMEC state in Figure \ref{fig:endcomponents}. However, due to the reachability constraints the agent might need to follow a policy that leaves a UMEC. We disallow such cases and make the following assumption to ensure that the agent does not leave UMECs.
	\begin{assumption} \label{assumption:unobservedmec}
		All unobserved maximal end components are closed.
	\end{assumption}

	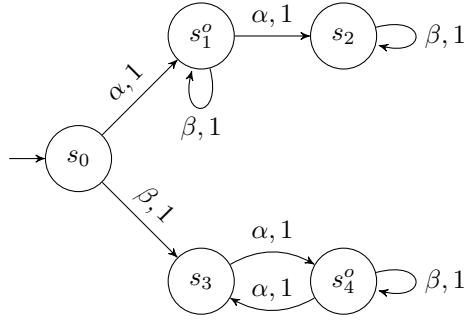
\begin{figure}[h] 
		\centering
		\begin{tikzpicture} 
		\node[state, initial]  (s0) {$s_0$};
		\node[state] [below right=of s0] (s3) {$s_3$};
		\node[state] [right=of s3] (s4) {$s_4^o$};
		\node[state] [above right=of s0 ] (s1) {$s_1^o$};
		\node[state] [right=of s1] (s2) {$s_2$};
		\draw 
		(s0) edge node[above, sloped] {$\alpha,1$} (s1)
		(s0) edge node[above, sloped] {$\beta,1$} (s3)
		(s1) edge node[above, sloped] {$\alpha,1$} (s2)
		(s2) edge[loop right] node[right] {$\beta,1$} (s2)
		(s3) edge[bend left] node[above, sloped] {$\alpha,1$} (s4)
		(s4) edge[bend left] node[above, sloped] {$\alpha,1$} (s3)
		(s1) edge[loop below] node[below] {$\beta,1$} (s1)
		(s4) edge[loop right] node[right] {$\beta,1$} (s4);
		\end{tikzpicture}
		\caption{An MDP with 5 states. A label $a,p$ of a transition refers to the transition that happens with probability $p$ when action $a$ is taken. The states marked with the superscript $o$ are observed.}
		\label{fig:endcomponents}
	\end{figure}

	\begin{remark}
		In the absence of Assumption \ref{assumption:unobservedmec}, to find the optimal stationary policy, one needs to check every UEC to determine whether the agents needs to stay or leave the UEC. Such a check increases computational complexity of finding a minimum-information admissible policy. For clarity of presentation, we here adopt Assumption \ref{assumption:unobservedmec}. In Appendix \ref{appendix:unobservedmec}, we investigate the more general problem without Assumption \ref{assumption:unobservedmec}.
	\end{remark}

	The optimal value of Problem \ref{problem} is 
	\begin{subequations}
		\label{pr:mininfleakagemodified}
		
		\begin{align}
		\inf \ &  \sum_{w \in W } x^{\pi}_{w} \iota^{\pi}_{w}  \label{eq:mininfrestime}
		\\
		\normalfont \text{s. t. } \ &  x^{\pi}_{s} = \sum_{a \in \mathcal{A}(s)} x^{\pi}_{s,a}, & & \forall s \in S \setminus C_{end} \label{cons:restimedefs}
		\\ 
		&  x^{\pi}_{s,a} \geq 0, & &
		\forall s \in S \setminus C_{end}, \ 
		\forall a \in \mathcal{A}(s) \label{cons:positiveactions}
		\\
		& \sum_{a \in \mathcal{A}(s)} x^{\pi}_{s,a} - \sum_{q \in S}  \sum_{a \in \mathcal{A}(q)} x^{\pi}_{q,a}\mathcal{P}_{q,a,s} = \mathds{1}_{s_0}(s),  & & \forall s \in S \setminus C_{end} \label{cons:floweqn}  
		\\
		& \sum_{q \in C_{reach}} \ \sum_{s \in S \setminus C_{end}}  \sum_{a \in \mathcal{A}(s)}  x^{\pi}_{s,a}\mathcal{P}_{s,a,q}  + \mathds{1}_{s_0}(q) \geq \nu_{reach}, \label{cons:reach}
		\end{align}
	\end{subequations}
	where the decision variables are $x^{\pi}_{s,a}$ for all $s \in S \setminus C_{end}$ and $a \in \mathcal{A}(s)$. The objective function \eqref{eq:mininfrestime} follows from Proposition \ref{proposition:replace} and the constraints \eqref{cons:restimedefs}-\eqref{cons:positiveactions} follow from definitions of expected residence times. The constraint \eqref{cons:floweqn} is the flow equation indicating that the expected number of arrivals into a state, i.e., the inflow, is equal to the expected number of departures from the state, i.e., the outflow. These equations ensure that there exists a policy that gives the computed  expected state-action residence times \cite{etessami2007multi}. The reachability constraint in \eqref{cons:reachreq} is equivalent to \eqref{cons:reach}.

	Note that some stationary admissible policies are infeasible for the optimization problem given in \eqref{pr:mininfleakagemodified}. In detail, the stationary policies that eventually always stay in an end component and visit an observed state infinitely often are infeasible. For instance, consider a policy $\pi$ such that $\Pr^{\pi}(Reach[s_2])= 0.5$ for the MDP given in Figure \ref{fig:endcomponents} with the reachability constraint $\Pr(Reach[s_2] \geq 0.5)$. While $\pi$ leads to infinite expected total information and satisfies the reachability constraint, it is not feasible for the problem in $\eqref{pr:mininfleakagemodified}$. One can easily check the existence of a policy that satisfies the reachability constraint via model checking tools such as \cite{kwiatkowska2002prism}. If there exists a policy that satisfies the task constraints, but the optimization problem given in \eqref{pr:mininfleakagemodified} is infeasible, we can say that the minimum-information admissible policy leaks infinite information.
	
	\begin{proposition} \label{proposition:infmin}
		If there exists a policy $\pi \in \Pi^{St}(\mathcal{M})$ that satisfies the reachability constraint given in \eqref{cons:reachreq}, then there exists a policy $\pi^* \in \Pi^{St}(\mathcal{M})$ that attains the optimal value of the optimization problem given in \eqref{pr:mininfleakagemodified}.
	\end{proposition}

	\begin{proposition} \label{proposition:cvx}
		The optimization problem given in \eqref{pr:mininfleakagemodified} is a convex optimization problem.
	\end{proposition}

	\begin{remark}
		After a preprocessing step that has polynomial-time complexity in the size of $\mathcal{M}$, the optimization problem can be formulated as a conic optimization problem which can be solved using interior-point methods \textnormal{\normalfont \cite{nesterov1994interior}} in polynomial-time in the size of $\mathcal{M}$.
	\end{remark}

	After computing the optimal expected state-action residence times by the optimization problem in \eqref{pr:mininfleakagemodified}, a stationary, minimum-information admissible policy can be synthesized using the relationship $x^{\pi}_{s,a} = \pi_{s,a} \ x^{\pi}_{s}$.

	\section{Bounds on the Estimation Error} \label{section:errorbound}
	In this section, we consider estimators for the transition probabilities at the observed states and derive the bounds on the expected estimation error in terms of MSE. Define $\sigma_w$ as the MSE of an unbiased estimator at a state $w$. We assume that, for the estimator at state $w$, the observed data are the whole path of the agent and the transition probabilities for the set $S\setminus \lbrace{ w \rbrace}$ of states are known.
	
	\begin{proposition} \label{proposition:infoisalowerbound}
	For an MDP $\mathcal{M}$ and a policy $\pi \in \Pi^{St}(\mathcal{M})$,   
	$$\sigma_w \geq \dfrac{\Pr^{\pi}(Reach[w])^2}{ x^{\pi}_{w} \iota^{\pi}_{w}}$$ for every state $w \in W$.
\end{proposition}

\begin{corollary} \label{corollary:totalmsebound} 
	For an MDP $\mathcal{M}$ and a policy $\pi \in \Pi^{St}(\mathcal{M})$,  the total MSE $\sum_{w \in W} \sigma_w$ satisfies $$\sum_{w \in W} \sigma_w \geq \dfrac{\underset{w \in W}{\min} \Pr^{\pi}(Reach[w])^2 |W|^2}{\mathbb{E} [ \iota^{\pi}_{W,\xi} ]}.$$ Consequently, if $\Pr^{\pi}(Reach[w]) = 1$ for every $\pi \in \Pi^{St}(\mathcal{M})$ and for all $w \in W$, then $\dfrac{|W|^2}{\mathbb{E}_{\xi} [ \iota^{\pi}_{W,\xi} ]}$ is a lower bound on the total MSE.
\end{corollary}

	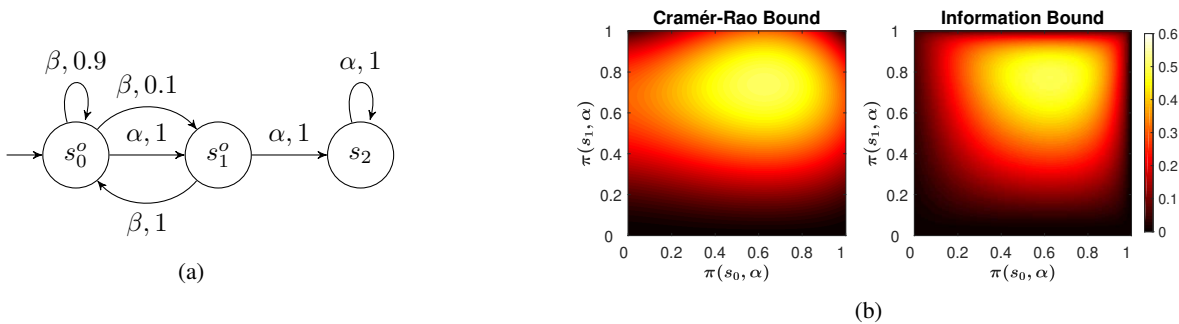
\begin{figure} [h]
		\centering
		\begin{subfigure}{0.4\paperwidth}
			\centering
			\begin{tikzpicture}
			\node[state, initial]  (q0) {$s_0^o$};
			\node[state] [right=of q0] (q1) {$s_1^o$};
			\node[state] [right=of q1] (t) {$s_2$};
			\draw 
			(q0) edge[loop above] node[above, sloped] {$\beta, 0.9$} (q0)
			(q0) edge[bend left=50] node[above, sloped] {$\beta, 0.1$} (q1)
			(q0) edge node[above, sloped] {$\alpha,1$} (q1)
			(q1) edge[bend left=50] node[below, sloped] {$\beta,1$} (q0)
			(q1) edge node[above, sloped] {$\alpha,1$} (t)
			(t) edge[loop above] node[above, sloped] {$\alpha,1$} (t);
			\end{tikzpicture}
			\caption{}
			\label{fig:msevsfishermdp}
		\end{subfigure}
		~ %add desired spacing between images, e. g. ~, \quad, \qquad, \hfill etc. 
		\begin{subfigure}{0.4\paperwidth}
			\centering
			\includegraphics[width=\textwidth]{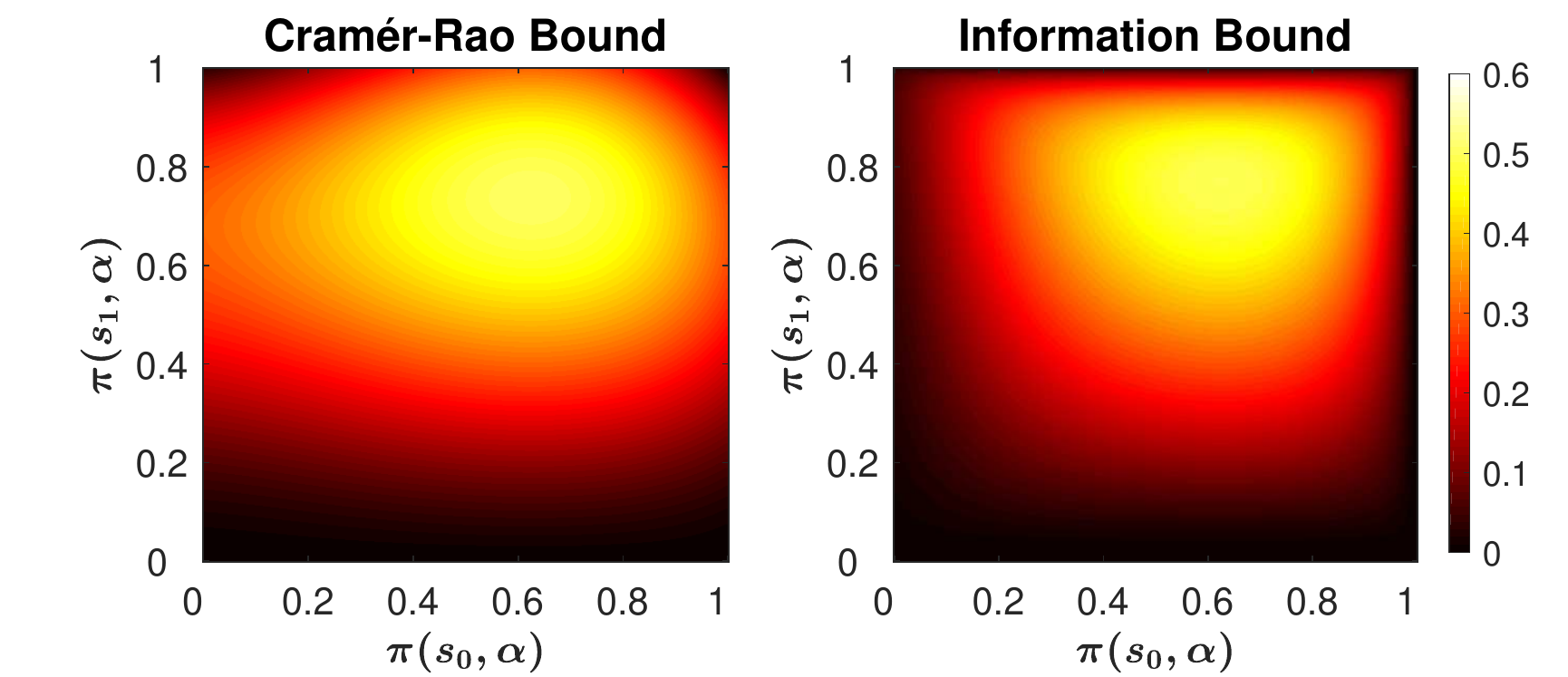}
			\caption{}
			\label{fig:msevsfishergraph}
		\end{subfigure}
		~ %add desired spacing between images, e. g. ~, \quad, \qquad, \hfill etc. 
		\caption{(a) An MDP with 3 states. A label $a,p$ of a transition refers to the transition that happens with probability $p$ when action $a$ is taken. The states marked with the superscript $o$ are observed. (b) The Cram\'er-Rao bound on the total MSE of the estimators and the error bound given in Corollary \ref{corollary:totalmsebound}. }\label{fig:msevsfisher} 
	\end{figure}
	
	An example of the bound given in Corollary \ref{corollary:totalmsebound} is illustrated in Figure \ref{fig:msevsfisher}. Both of the observed states are visited under any stationary policy and the reciprocal of the expected total information is directly a lower bound on the total MSE of the estimators. One who wants to maximize the total MSE of the estimators may prefer to optimize over the expected total information instead of the Cram\'er-Rao bound since the Cram\'er-Rao is not a convex or concave function of the expected residence time parameters while the minimum-information admissible policy can be computed via a convex optimization problem. 
	\section{Numerical Examples} \label{section:examples}
	In this section, we illustrate the proposed method through two numerical examples. We solved the optimization problems using CVX toolbox \cite{cvx} with MOSEK \cite{andersen2000mosek} on a computer with an Intel Core i7-8550u 1.8 GHz CPU and 8 GB of RAM.  
	
	\subsection{Partly Hidden Agent}
	In this example, we explain the characteristics of the minimum-information admissible policy through different scenarios. 
	
	The environment which is given in Figure \ref{fig:quadrantexample} consists of 4 regions that are separated with walls and connected to each other with bridges. Each region is a $20 \times 20$ grid world and each tile in these regions represents a state. Except for the reach state, the agent has $4$ actions, namely, up, down, left, and right, at every state. When the agent takes an action the transition happens into the target direction with probability $0.8$ and in the other directions uniformly randomly with probability $0.2$. If a direction is out of the grid the transition probability to that direction is proportionally distributed to the other directions.
	
	The initial state is the black top-left corner tile and the reach state is the green bottom-middle tile. The task of the agent is to reach the reach state with probability $1$. While the agent is in the gray tiles, the observer cannot observe the transitions of the agent.  
	
	\begin{figure} [h]
		\centering
		\begin{subfigure}[b]{0.27\paperwidth}
			\centering
			\includegraphics[width=\textwidth]{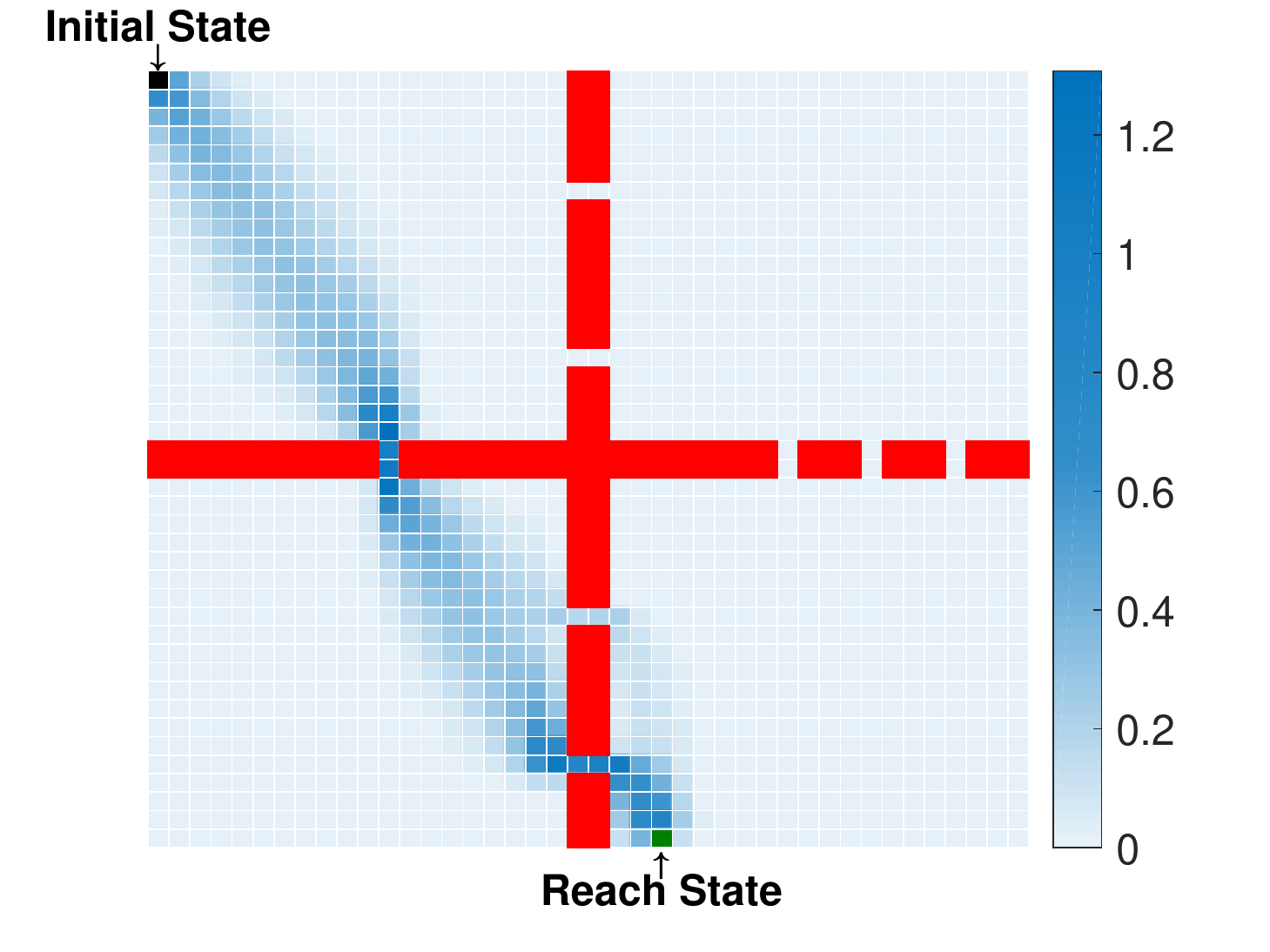}
			\caption{}
			\label{fig:quadrantmil1}
		\end{subfigure}
		\begin{subfigure}[b]{0.27\paperwidth}
			\centering
			\includegraphics[width=\textwidth]{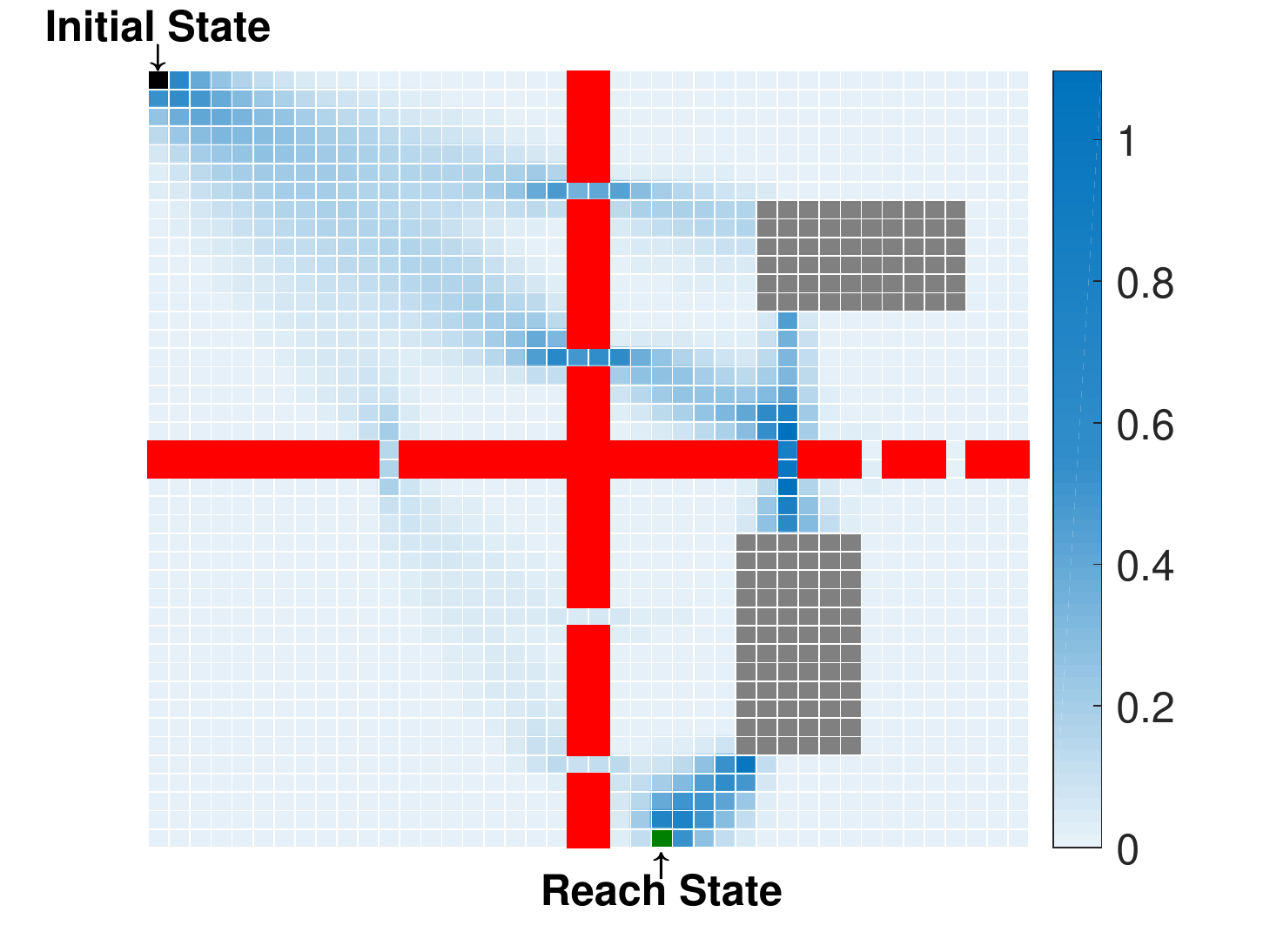}
			\caption{}
			\label{fig:quadrantmil2}
		\end{subfigure}
		\begin{subfigure}[b]{0.27\paperwidth}
			\centering
			\includegraphics[width=\textwidth]{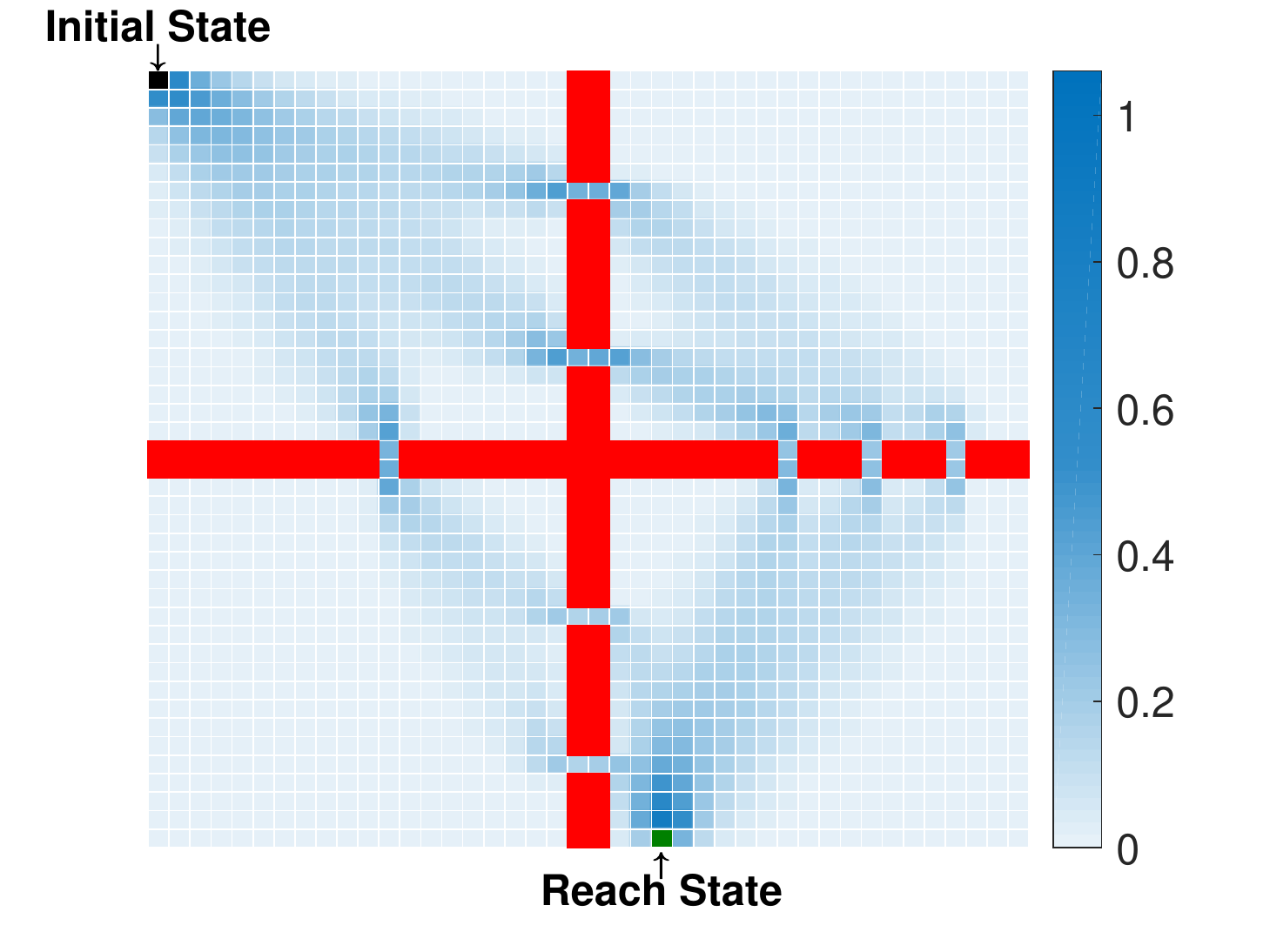}
			\caption{}
			\label{fig:quadrantmil3}
		\end{subfigure}
		\caption{The heatmaps of expected state residence times for partly hidden agent example. For the scenario given in Figure \ref{fig:quadrantmil2} the environment has some unobserved regions while every state is observed for the scenario given in Figure \ref{fig:quadrantmil2}. The scenario given in Figure \ref{fig:quadrantmil3} considers exit information of in addition to the transition information. } \label{fig:quadrantexample}
	\end{figure}

	In the first scenario (see Figure \ref{fig:quadrantmil1}) all states are observed except the reach state and the bridge states. The agent completes the task with a low number of observed transitions (See Table \ref{table:partlyhiddenagent}) with randomized transitions. Note that the randomization only happens between the states that are in the direction of the reach state since further randomization leads to more observations. When the unobserved regions are present in the environment (see Figure \ref{fig:quadrantmil2}), the policy generates paths that pass through the unobserved regions to reduce the number of observations. However, the unobserved regions are not always utilized. For example, in the top-right region if the agent is already away from the unobserved region, it directly goes to the bottom-right region. Although, no information leaks in the unobserved regions, the agent leaks information during the process of reaching those states. 
	
	We remark that the minimum-information admissible policy minimizes only the information of transitions from the observed states. While this approach reduces the amount of leaked information in the local sense, i.e., the transitions between the states, the global behavior, i.e., the transitions between the regions, might be easily inferred. We observe such a phenomenon for the scenario given in Figure \ref{fig:quadrantmil1}; the agent leaves the regions using the same bridge. This behavior may be risky if there is an adversary that is interested in the information of which bridge is used. To avoid this behavior, we add a weighted penalty, \textit{exit information}, for each region. The exit information of a region has the same form with the transition information and consists of the expected state residence times of the bridges. With the exit information (see Figure \ref{fig:quadrantmil3})  the agent randomizes its exit bridge from the regions compared to the initial case (see Figure \ref{fig:quadrantmil1}).

	\begin{table}[h] 
		\centering
		\caption{Numerical values for partly hidden agent example.}
		\label{table:partlyhiddenagent}
		\begin{tabular}{@{}cccc@{}}
			\toprule
			Scenario & \begin{tabular}[c]{@{}c@{}}Expected \\ Total Information\end{tabular} & \begin{tabular}[c]{@{}c@{}}Expected Number\\  of Observations\end{tabular} & Solving Time \\ \midrule
			Figure \ref{fig:quadrantmil1} &        152.20                                                               &                   81.87                                                         &        0.52      \\
			Figure \ref{fig:quadrantmil2} &           146.00                                                             &                                                    78.89                        &      0.38        \\
			Figure \ref{fig:quadrantmil3} &             179.64                                                          &                             98.43                                               &    0.58          \\ \bottomrule
		\end{tabular}
	\end{table}
	
	\subsection{Inference of Local Behavior}
	We explain the difference between the proposed method and the policy synthesis via entropy maximization through this example. The environment is a $11 \times 11$ grid world given in Figure \ref{fig:gridworldentropy} where each tile represents a state. The black tile is the initial state, the green tile is the reach state, and the red tiles are the absorbing states. Except for the absorbing states and the reach state the agent can transition to $4$ directions, namely, up, down, left, and right, at every state. When the agent takes an action, the transition happens in the target direction with probability $1$. If a direction is out of the grid the action is not allowed. The task of the agent is to reach the reach state with probability $1$.  
	
	We compare the policies in terms of their estimation error, which is calculated for different number of sample paths. The observer gets sample paths and estimates the transition probabilities at the observed states using a sample mean estimator. We measure the estimation error for a state by the mean squared error (MSE) between the observed and actual transition distributions at the observed states. The total error is the sum of MSE for each state. If there is no observation sample from a state, we set the MSE for that state.For the weighted MSE error, the weight of a state is ratio between the number of observations from the state and the total number of observations. 
	
	Maximizing the entropy of an MDP is equivalent to maximizing the entropy of the possible paths, and a high entropy value leads to unpredictable paths. Under the reachability constraint, the maximum entropy of the MDP given in Figure \ref{fig:gridworldentropy} is unbounded. For policy synthesis, we follow the procedure given in \cite{savas2018entropy} and impose an upper bound on the expected total state residence time $\Gamma$. As the bound increases, the maximum entropy value of the MDP increases. We synthesize three policies that maximizes the entropy of MDP with different values for $\Gamma = 15, 60,$ and $120$. 
	
	\begin{figure} [t]
		\centering
		\begin{subfigure}[b]{0.19\paperwidth}
			\centering
			\includegraphics[width=\textwidth]{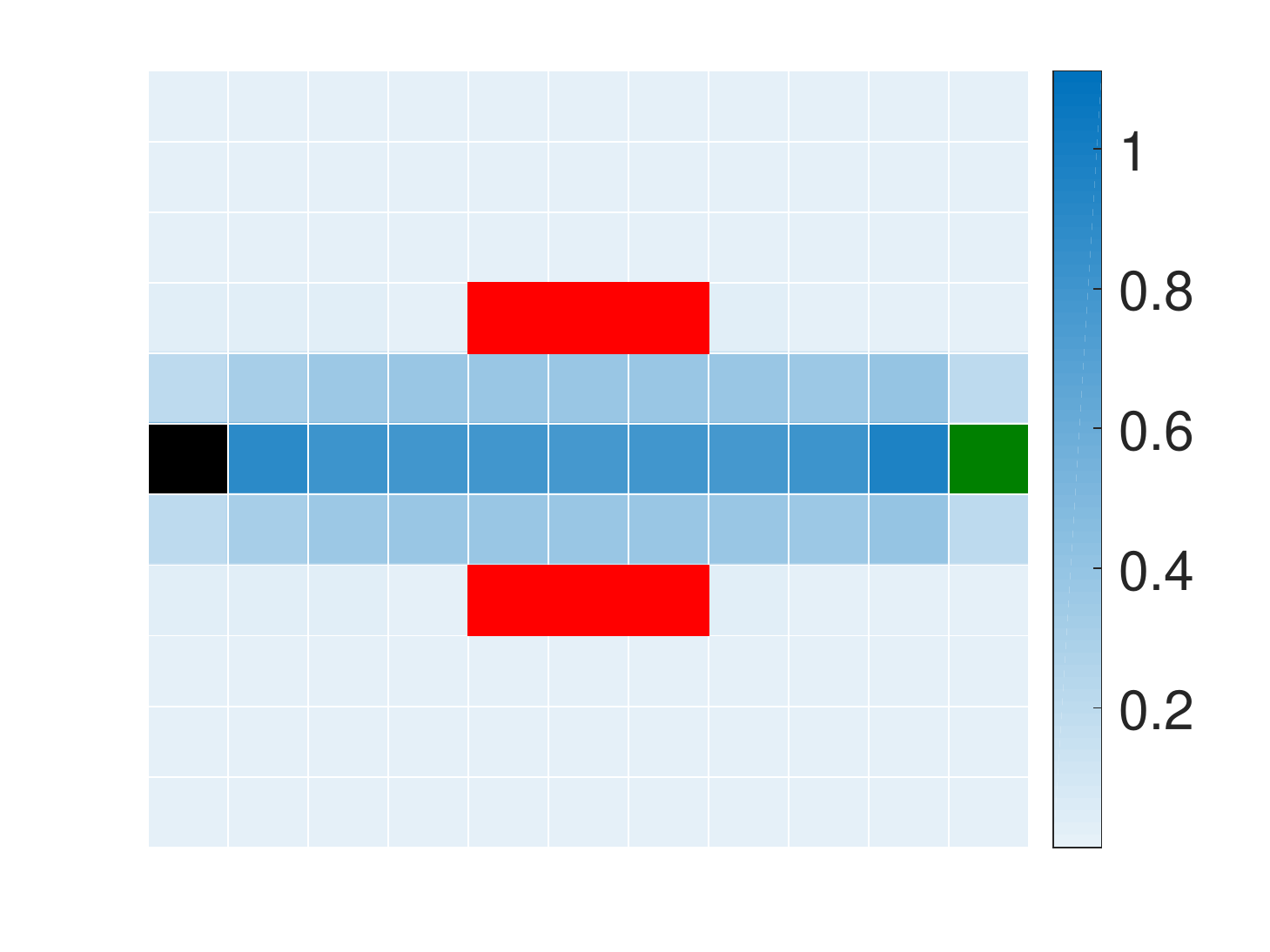}
			\caption{The minimum-information admissible policy}
			\label{fig:gridworld2mininfo}
		\end{subfigure}
		\begin{subfigure}[b]{0.19\paperwidth}
			\centering                 
			\includegraphics[width=\textwidth]{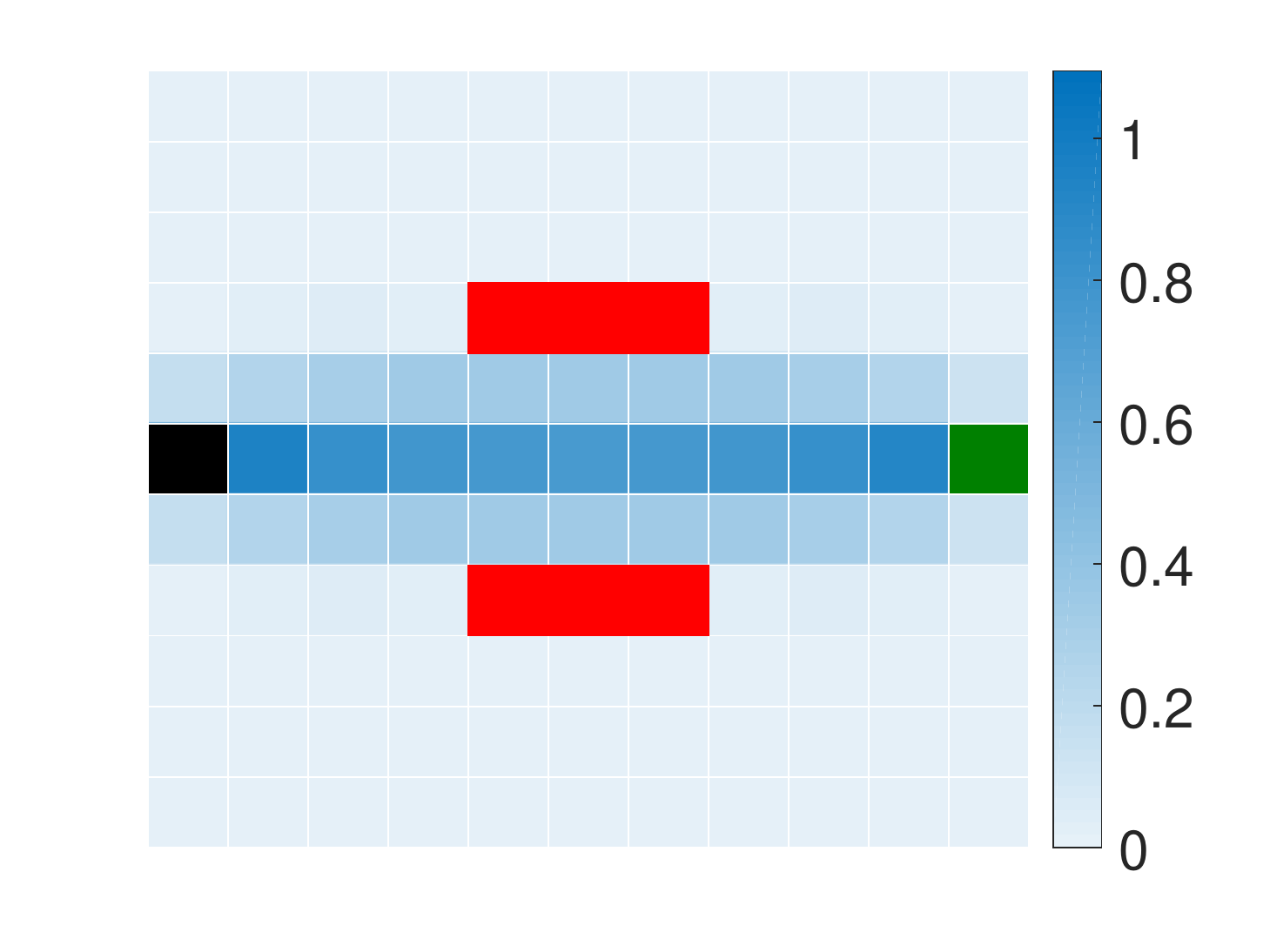}
			\caption{The maximum-entropy policy ($\Gamma = 15$)}
			\label{fig:gridworld2entropy15}
		\end{subfigure}
		\begin{subfigure}[b]{0.19\paperwidth}
			\centering
			\includegraphics[width=\textwidth]{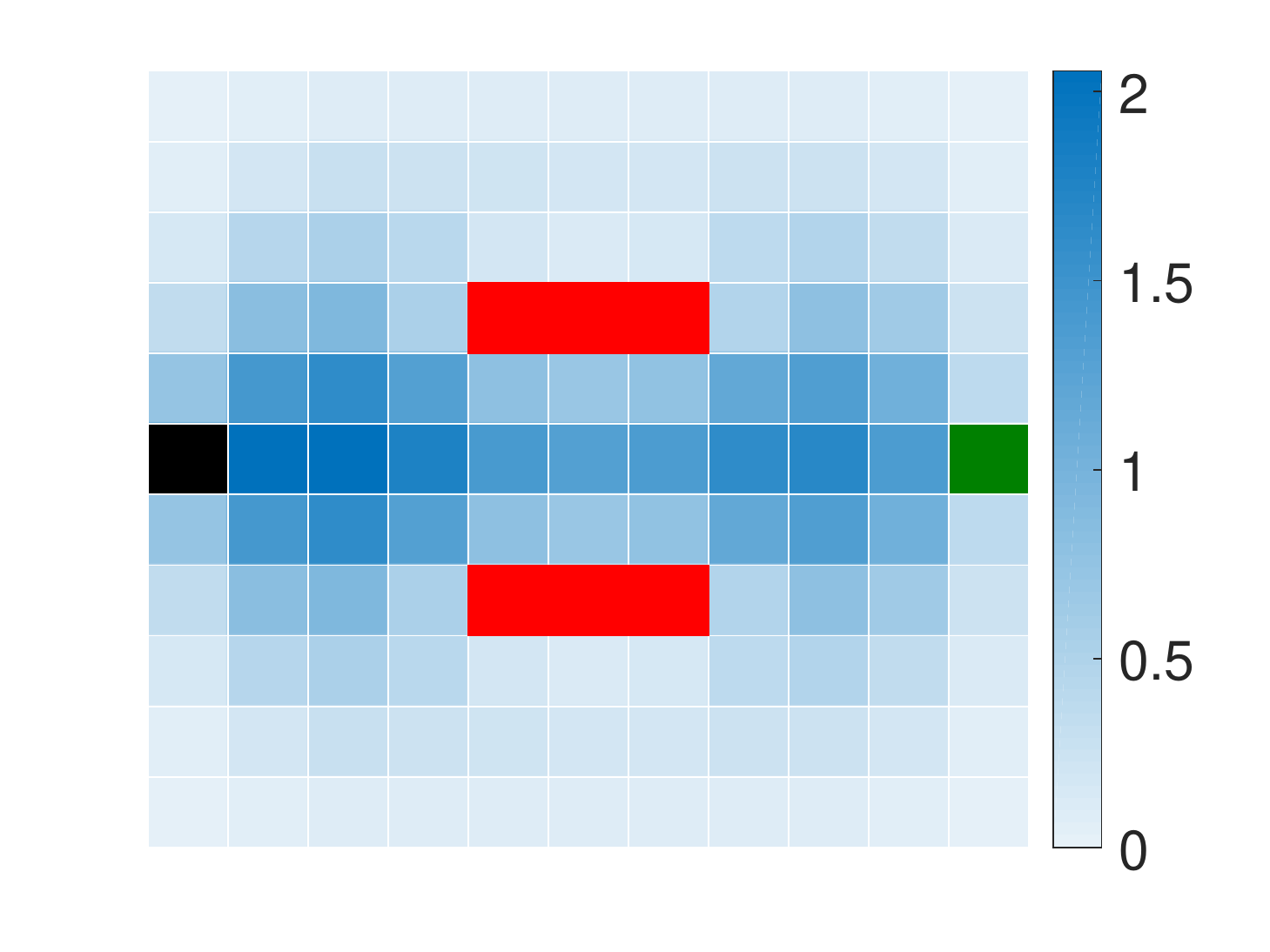}
			\caption{The maximum-entropy policy ($\Gamma = 60$)}
			\label{fig:gridworld2entropy60}
		\end{subfigure}
		\begin{subfigure}[b]{0.19\paperwidth}
			\centering
			\includegraphics[width=\textwidth]{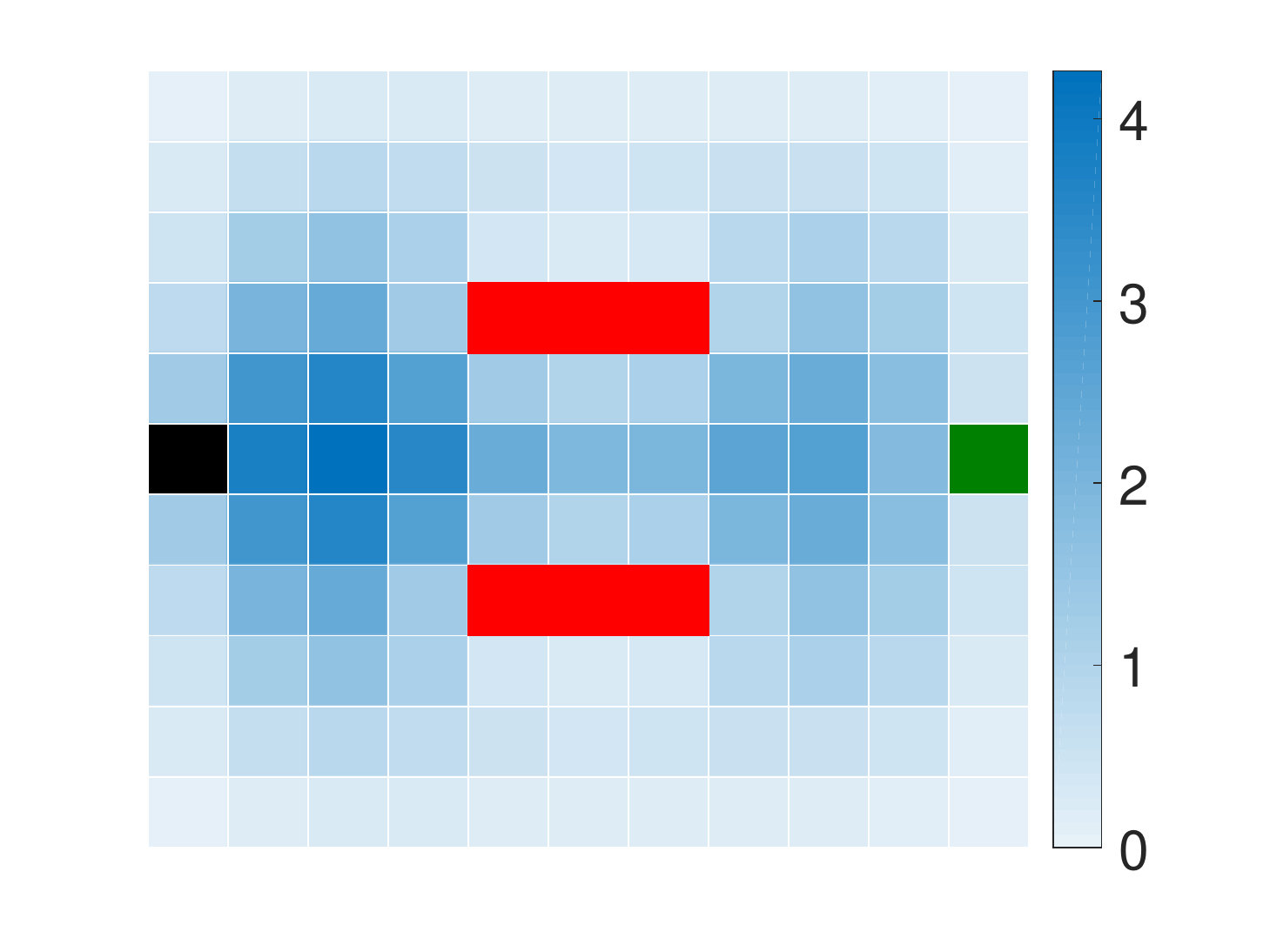}
			\caption{The maximum-entropy policy ($\Gamma = 120$)}
			\label{fig:gridworld2entropy120}
		\end{subfigure}
		\caption{The expected state residence times for inference of local behavior example.} \label{fig:gridworldentropy}
	\end{figure}

	For low values of $\Gamma$ such as $15$, the minimum-information admissible and the maximum-entropy policies show similar behavior. However, for the high values of $\Gamma$, the difference between the minimum-information admissible policy and the maximum-entropy policy becomes clear. The minimum-information admissible policy completes the task with a low number of non-informative observations. On the other hand, the maximum-entropy policy visits the observed states more to explore more paths and randomize the probabilities of paths. While the agent follows different paths, the expected residence times at the observed states increases and observer gets more samples. Although the policy is randomized and samples are less informative, transition probabilities are inferred due to the high number of observations. The result suggests that the unpredictability of the paths does not imply the limitation of inference for the transitions between states. Hence, the minimum-information admissible policy and the maximum-entropy policy serve different purposes.

		\begin{figure} [h]
		\centering
		\begin{subfigure}{0.4\paperwidth}
	\centering
	\includegraphics[width=\textwidth]{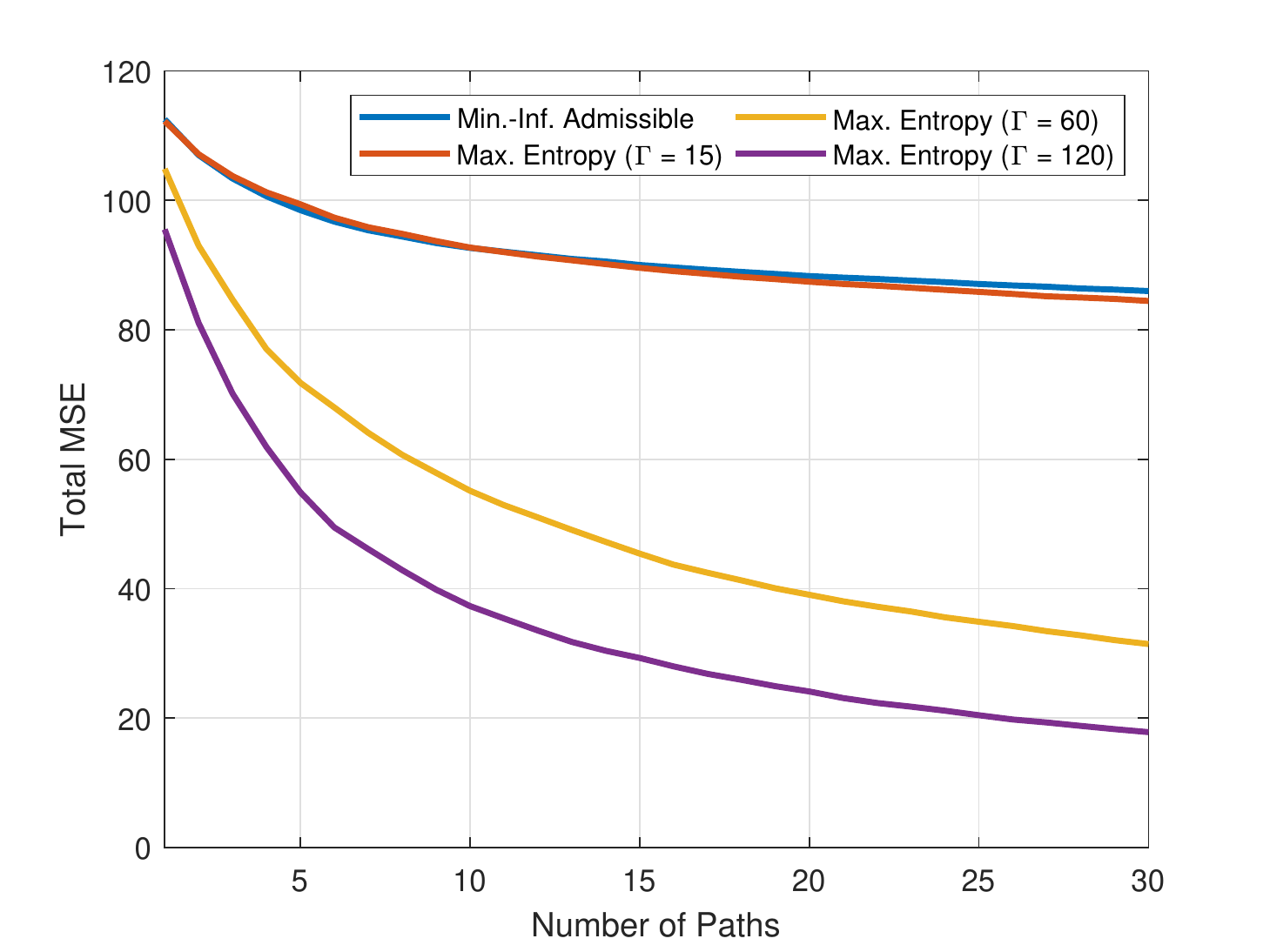}
	\caption{}
	\label{fig:msevsfishergraph1}
	\end{subfigure}
		~ %add desired spacing between images, e. g. ~, \quad, \qquad, \hfill etc. 
		\begin{subfigure}{0.4\paperwidth}
			\centering
			\includegraphics[width=\textwidth]{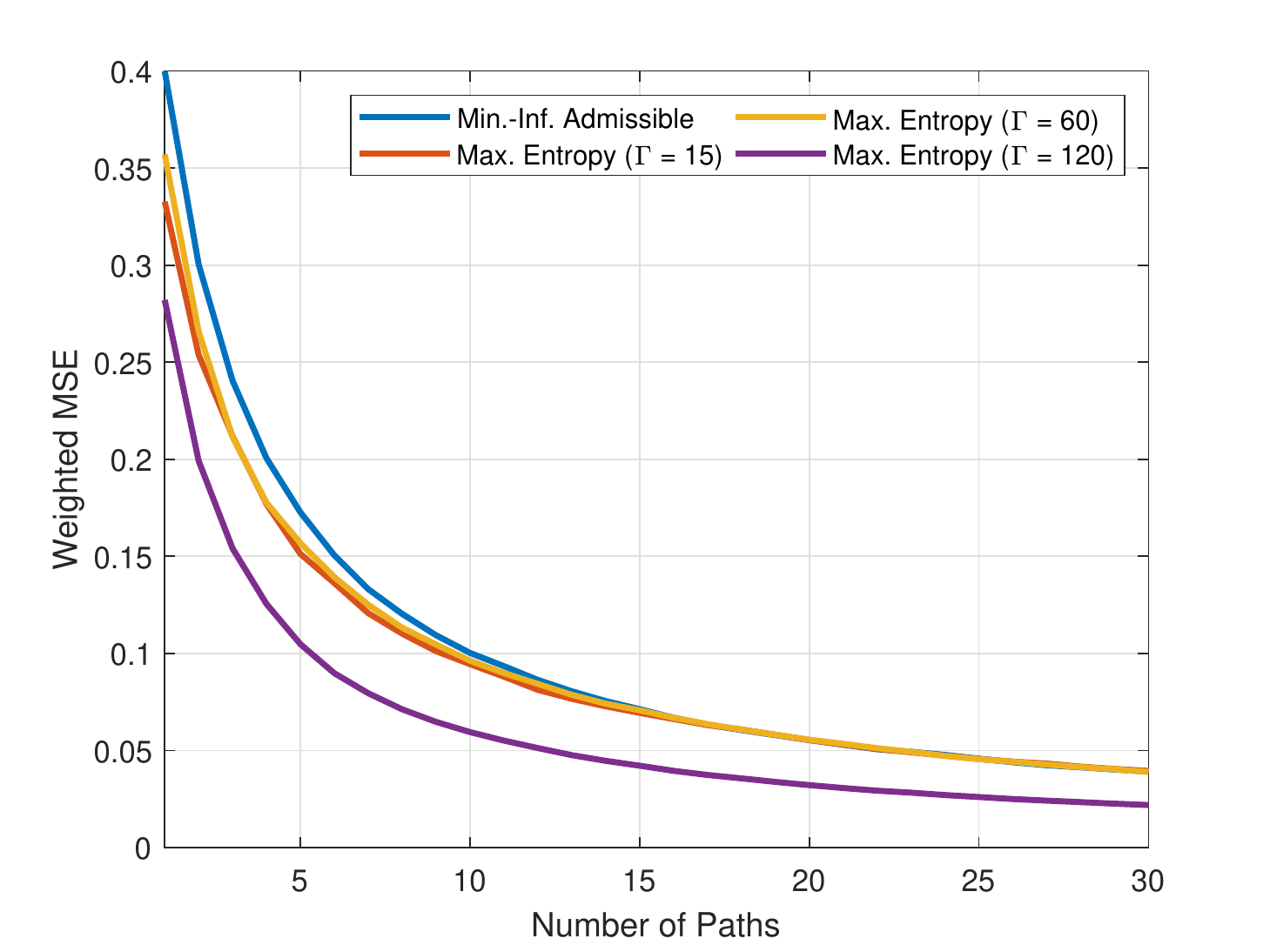}
			\caption{}
			\label{fig:msevsfishergraph2}
		\end{subfigure}
		~ %add desired spacing between images, e. g. ~, \quad, \qquad, \hfill etc. 
		\caption{The expected estimation errors. The curves are averaged over 100 experiments.}
	\end{figure}

	\section{Conclusion} \label{section:conclusion}
	We focus on policy synthesis for an agent whose behavior is inferred by an outside adversarial observer. Such an agent must as less informative observations as possible to the observer while completing its task. Based on this criterion, we introduced transition information which is based on the Fisher information and measures the amount of information leaked to the observer from a transition. Then, we formulated a problem that minimizes the expected total information leaked to the observer and showed the existence of such a policy. The significant feature of the proposed method is that it balances a possible trade-off between the number of observations and the informativeness of each observation.  
	
	The proposed method relies on the assumption that the agent follows a stationary policy on the observed states. A history dependent planning method may deceive the observer by actively changing the policy. We aim to remove this assumption and design an algorithm that takes the past transitions into account. 
	
\section*{Acknowledgement}
	This work was supported in part by DARPA W911NF-16-1-0001.
	
	\bibliographystyle{unsrt}
	\bibliography{ref}		

\begin{thebibliography}{10}

\bibitem{frieden2004science}
B.~R. Frieden.
\newblock {\em Science from Fisher information: A Unification}.
\newblock Cambridge University Press, 2004.

\bibitem{alpcan2015information}
T.~Alpcan and I.~Shames.
\newblock An information-based learning approach to dual control.
\newblock {\em IEEE Transactions on Neural Networks and Learning Systems},
  26(11):2736--2748, 2015.

\bibitem{emery1998optimal}
A.~F. Emery and A.~V. Nenarokomov.
\newblock Optimal experiment design.
\newblock {\em Measurement Science and Technology}, 9(6):864--876, 1998.

\bibitem{farokhi2017optimal}
F.~Farokhi and H.~Sandberg.
\newblock Optimal privacy-preserving policy using constrained additive noise to
  minimize the {Fisher} information.
\newblock In {\em 56th IEEE Conference on Decision and Control}, pages
  2692--2697, 2017.

\bibitem{farokhi2017fisher}
F.~Farokhi and H.~Sandberg.
\newblock Fisher information as a measure of privacy: Preserving privacy of
  households with smart meters using batteries.
\newblock {\em IEEE Transactions on Smart Grid}, 9(5):4726--4734, 2018.

\bibitem{agmon2008multi}
N.~Agmon, S.~Kraus, and G.~A. Kaminka.
\newblock Multi-robot perimeter patrol in adversarial settings.
\newblock In {\em IEEE International Conference on Robotics and Automation},
  pages 2339--2345, 2008.

\bibitem{paruchuri2006security}
P.~Paruchuri, M.~Tambe, F.~Ord\'{o}\~{n}ez, and S.~Kraus.
\newblock Security in multiagent systems by policy randomization.
\newblock In {\em Joint Conference on Autonomous Agents and Multiagent
  Systems}, pages 273--280, 2006.

\bibitem{savas2018entropy}
Y.~Savas, M.~Ornik, M.~Cubuktepe, and U.~Topcu.
\newblock Entropy maximization for {Markov} decision processes under temporal
  logic constraints.
\newblock {\em arXiv preprint arXiv:1807.03223 [math.OC]}, 2018.

\bibitem{lehmann2006theory}
E.~L. Lehmann and G.~Casella.
\newblock {\em Theory of Point Estimation}.
\newblock Springer, 2nd edition, 1998.

\bibitem{etessami2007multi}
K.~Etessami, M.~Kwiatkowska, M.~Y. Vardi, and M.~Yannakakis.
\newblock Multi-objective model checking of {Markov} decision processes.
\newblock In {\em International Conference on Tools and Algorithms for the
  Construction and Analysis of Systems}, pages 50--65, 2007.

\bibitem{kwiatkowska2002prism}
M.~Kwiatkowska, G.~Norman, and D.~Parker.
\newblock {PRISM}: Probabilistic symbolic model checker.
\newblock In {\em International Conference on Modelling Techniques and Tools
  for Computer Performance Evaluation}, pages 200--204, 2002.

\bibitem{nesterov1994interior}
Yu. Nesterov and A.~Nemirovskii.
\newblock {\em Interior-Point Polynomial Algorithms in Convex Programming}.
\newblock Society for Industrial and Applied Mathematics, 1994.

\bibitem{cvx}
M.~Grant and S.~Boyd.
\newblock {CVX}: Matlab software for disciplined convex programming, version
  2.1.
\newblock \url{http://cvxr.com/cvx}, 2014.

\bibitem{andersen2000mosek}
MOSEK ApS.
\newblock The {MOSEK} optimization toolbox for matlab manual, version 8.1.
\newblock \url{http://docs.mosek.com/8.1/toolbox/index.html}, 2017.

\bibitem{boyd2004convex}
S.~Boyd and L.~Vandenberghe.
\newblock {\em Convex Optimization}.
\newblock Cambridge University Press, 2004.

\bibitem{zegers2015fisher}
P.~Zegers.
\newblock Fisher information properties.
\newblock {\em Entropy}, 17(7):4918--4939, 2015.

\end{thebibliography}
	
	\appendices
	\section{Unobserved Maximal End Components} \label{appendix:unobservedmec}
	In Section \ref{section:methodology}, we said that after reaching an unobserved maximal end component (UMEC), the agent may leak no more information since there exists a stationary policy that always stays in the UMEC. However, such a policy may not be admissible due to the reachability constraint. In that case, the agent has to leave the UMEC. 
	
	Assumption \ref{assumption:unobservedmec} ensures that the agent cannot leave UMECs. Every policy stays in UMECs and hence the outflow from these states is zero. Thanks to this assumption, we only need to consider the policies where the agent stays in UMECs and synthesize the policy accordingly.
	
	In the following subsections we investigate the cases where the assumption does not hold. Appendix \ref{appendix:stationarymec} provides an exhaustive search algorithm to find the optimal stationary policy. Appendix \ref{appendix:nonstationarymec} provides an algorithm that searches a different class of policies to find the optimal policy.
	
	\subsection{Agents with Stationary Policies} 
\label{appendix:stationarymec}
Consider the MDPs given in Figure \ref{fig:UECs} where the reachability requirement is $\Pr^{\pi}(Reach[s_4 \cup s_5 ]) \geq 0.5$. For both MDPs information is leaked only at state $s_3$ and it is proportional to the expected residence time at state $s_3$, i.e., $x^{\pi}_{s_3}$. Note that also the reachability probability is equal to the expected residence time at state $s_3$, i.e.,  $x^{\pi}_{s_3} = \Pr^{\pi}(Reach[s_4 \cup s_5])$.

		\begin{figure}[h] 
	\centering
	\begin{subfigure} [h]{0.4\paperwidth} 
		\centering
		\begin{tikzpicture}[scale=0.25] 
		\node[state, initial]  (s0) {$s_0$};
		\node[state] [right=of s0] (s1) {$s_1$};
		\node[state] [above=of s1] (s2) {$s_2$};
		\node[state] [right=of s1] (s3) {$s_3^o$};
		\node[state] [above right=of s3] (s4) {$s_4$};
		\node[state] [below right=of s3] (s5) {$s_5$};
		\draw 
		(s0) edge node[above, sloped] {$\alpha,1$} (s1)
		(s1) edge[bend left] node[above, sloped] {$\alpha,1$} (s2)
		(s2) edge[bend left] node[above, sloped] {$\alpha,1$} (s1)
		(s1) edge node[above, sloped] {$\beta,1$} (s3)
		(s3) edge node[above, sloped] {$\alpha,0.5$} (s4)
		(s3) edge node[below, sloped] {$\alpha,0.5$} (s5)
		(s4) edge[loop right] node[above] {$\alpha,1$} (s4)
		(s5) edge[loop right] node[above] {$\alpha,1$} (s5);
		\end{tikzpicture}
		\caption{}
		\label{fig:oneec}
	\end{subfigure}
	\begin{subfigure} [h]{0.4\paperwidth}  
		\centering
		\begin{tikzpicture}[scale=0.25]
		\node[state, initial]  (s0) {$s_0$};
		\node[state] [right=of s0] (s1) {$s_1$};
		\node[state] [above=of s1] (s2) {$s_2$};
		\node[state] [right=of s1] (s3) {$s_3^o$};
		\node[state] [above right=of s3] (s4) {$s_4$};
		\node[state] [below right=of s3] (s5) {$s_5$};
		\draw 
		(s0) edge node[above, sloped] {$\alpha,1$} (s1)
		(s1) edge[bend left] node[above, sloped] {$\alpha,1$} (s2)
		(s2) edge[bend left] node[above, sloped] {$\alpha,1$} (s1)
		(s1) edge node[above, sloped] {$\beta,1$} (s3)
		(s2) edge[loop above] node[above, sloped] {$\beta,1$} (s2)
		(s3) edge node[above, sloped] {$\alpha,0.5$} (s4)
		(s3) edge node[below, sloped] {$\alpha,0.5$} (s5)
		(s4) edge[loop right] node[above] {$\alpha,1$} (s4)
		(s5) edge[loop right] node[above] {$\alpha,1$} (s5);
		\end{tikzpicture}
		\caption{}
		\label{fig:twoec}
	\end{subfigure}
	\caption{MDPs with 6 states. A label $a,p$ of a transition refers to the transition that happens with probability $p$ when action $a$ is taken. The states marked with the superscript $o$ are observed.}
	\label{fig:UECs}
\end{figure}
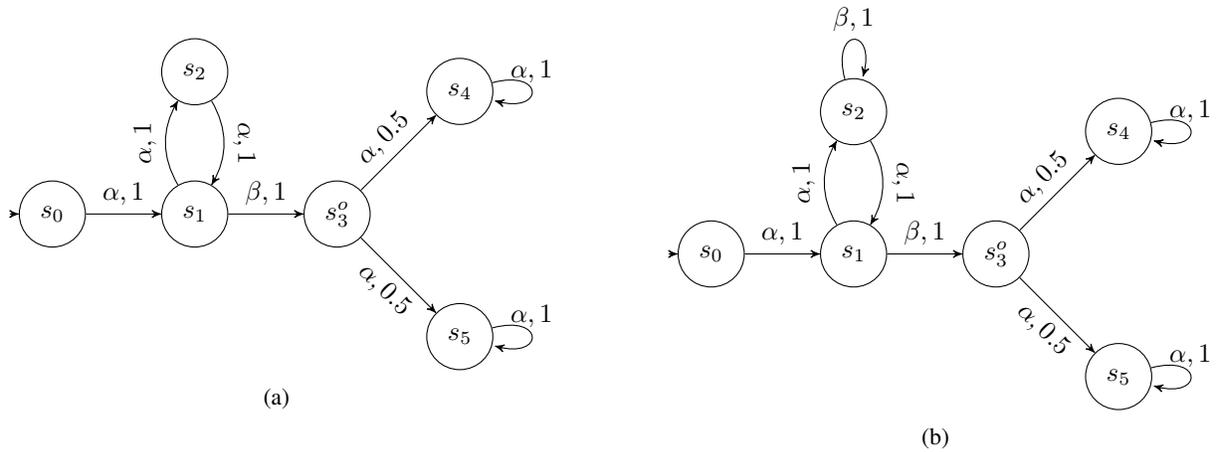

One might naturally think that a policy that makes $x^{\pi}_{s_3} = 0.5$ is a minimum-information admissible policy. However, we note that such a stationary policy might not exist since $x^{\pi}_{s_3}$ and $\Pr^{\pi}(Reach[s_4 \cup s_5])$ are not continuous functions of stationary policies. For the MDP given in Figure \ref{fig:oneec}, a stationary policy $\pi$ has $\Pr^{\pi}(Reach[s_4 \cup s_5]) = 1$ if $\pi_{s_1,\beta} > 0$ and $\Pr^{\pi}(Reach[s_4 \cup s_5]) = 0$ otherwise. Every policy $\pi^*$ such that $\pi^*(s_1,\beta) > 0$ is a stationary, minimum-information admissible policy. However, such a policy does not satisfy the reachability requirement with equality. For the MDP given in Figure \ref{fig:twoec}, it is possible to find a stationary policy that satisfies the reachability requirement with equality. The stationary policy $\pi^*$ with $\pi^*_{s_1,\alpha} = 0.5$, $\pi^*_{s_1,\beta} = 0.5$, $\pi^*_{s_2,\alpha} = 0$, and $\pi^*_{s_2,\beta} = 1$ is the minimum-information admissible policy.

	By the MDPs given in Figure \ref{fig:UECs}, we note that determining whether the optimal policy stays in a UEC is not trivial. To find the stationary, minimum-information admissible policy, we give an optimization algorithm that is based on exhaustive search of all unobserved end components.  
	
	\begin{definition}
		A \textit{union unobserved end component} is a sub-MDP $(C,D)$ that is union of UECs $(C_1, D_1), \ldots, (C_N, D_N)$ such that $C = C_1 \cup \ldots \cup C_N$ and $D(s) = D_1(s) \cup \ldots \cup D_N(s)$ for every $s$ in $C$.
	\end{definition}
	
	Algorithm \ref{alg:combinatorial} takes a subset of UMEC states, checks whether this subset is a union unobserved end component (see Lines \ref{algstep:takeSet}-\ref{algstep:checkSet}). If the subset is a union unobserved end component, it finds the optimal stationary policy that makes the agent stay in the union unobserved end component (see Lines \ref{algstep:setUEC} - \ref{algstep:synpol}). The algorithm outputs the minimum-information admissible policy after checking all subsets. 
	  
	\begin{algorithm} [t]
		\caption{Synthesis of a stationary, minimum-information admissible policy for MDPs with UMECs - Process 1 } \label{alg:combinatorial}
		\begin{algorithmic}[1]
			\State \textbf{Input:} An MDP $\mathcal{M} = (S, \mathcal{A}, \mathcal{P}, s_0)$, the set of observed states $W$, the set of states to be reached $C_{reach}$, and the reachability probability $\nu_{reach}$.
			\State \textbf{Output:} A stationary, minimum-information admissible policy $\pi^*$ for $\mathcal{M}$.
			\State $R:= \emptyset$.
			\State Find every UMEC $(C,D)$ and set $R:= R \cup C$.
			\State $L:= 2^R.$
			\State $minval := \infty$
			\ForAll {$l \in L$} \label{alg:checkconfigurations} \label{algstep:takeSet}
				\If{$l$ is a union unobserved end component}  \label{algstep:checkSet}
				\State $C_{end} := l$ \label{algstep:setUEC}
				\State Solve \eqref{pr:mininfleakagemodified} with $C_{end}$, $C_{reach}$, and $\nu_{reach}$. Set the optimal value to $val$ and set the solution to $restimes$. \label{algstep:synpol}
				\If {$val \leq minval$}
				\State $minval:= val$, $minset := l$, $minrestimes := restimes$.
				\EndIf
				\EndIf

			\EndFor
			\State $C_{end} := minset$.
			\State Synthesize the minimum-information admissible $\pi^*$ policy using Algorithm \ref{alg:polsyn2} with  $minrestimes$ and $C_{end}$.
			\State \textbf{return} $\pi^*$
		\end{algorithmic}
	\end{algorithm}
		\begin{algorithm} [t]
		\caption{Synthesis of a stationary, minimum-information admissible policy for UMECs - Process 2} \label{alg:polsyn2}
		\begin{algorithmic}[1]
			\State \textbf{Input:} An MDP $\mathcal{M} = (S, \mathcal{A}, \mathcal{P}, s_0)$, the expected residence times $x^{\pi^*}_{s,a}$ for $\mathcal{M}$, and $C_{end}$.
			\State \textbf{Output:} A stationary, minimum-information admissible policy $\pi^*$ for $\mathcal{M}$.
			\State Synthesize a policy $\pi^{stay}$ such that for a state $s \in C_{end}$, $\sum_{q \in C_{end}} \sum_{a \in \mathcal{A}(s)} \pi^{stay}_{s,a} \ \mathcal{P}_{s,a,q} =1$.
			\ForAll {$s \in S$}
			\If {$\sum_{a \in \mathcal{A}(s)} x^{\pi^*}_{s,a}=0$}
			\ForAll {$a \in \mathcal{A}(s)$}
			\State Set $\pi^*_{s,a}$ arbitrarily between 0 and 1 subject to $\sum_{a' \in \mathcal{A}(s)} = \pi^*_{s,a'} = 1$.
			\EndFor
			\ElsIf {$s \in C_{end}$} 
			\ForAll {$a \in \mathcal{A}(s)$}
		\State $\pi^*_{s,a} := \pi^{stay}_{s,a}$.
			\EndFor
			\Else
			\ForAll {$a \in \mathcal{A}(s)$}
			\State $\pi^*_{s,a}:= \dfrac{x^{\pi^*}_{s,a}}  {\sum_{a' \in \mathcal{A}(s)} x^{\pi^*}_{s,a'}}$ .
			\EndFor
			\EndIf
			\EndFor
			\State \textbf{return} $\pi^*$.
		\end{algorithmic}
	\end{algorithm}

	\begin{remark}
		Note that the size of $R$ is $O(|S|)$ in Algorithm \ref{alg:combinatorial} and the size of $L$ is $O(2^{|S|})$. Checking whether a set of states $S$ is a union unobserved end component has $O(|S|^3|\mathcal{A}|)$ complexity. Hence, the exhaustive search given in Algorithm \ref{alg:combinatorial} increases the complexity by a factor of $O(2^{|S|} |S|^3 |A|)$.
	\end{remark}
	
	\subsection{Agents with Nonstationary Policies} 
 \label{appendix:nonstationarymec}
In this section, we remove Assumption \ref{assumption:unobservedmec} and introduce an algorithm that avoids the exhaustive search given in Algorithm \ref{alg:combinatorial}. The exhaustive search is required as a drawback of stationary policies. We extend the policy space of the agent to find the optimal policy with lower computational complexity by allowing the agent to pick a policy that might be nonstationary for the unobserved states. We call a policy $\pi$ observation stationary if it is stationary at the observed states and define $\Pi^{Obs \ St}(\mathcal{M})$ as set of the observation stationary policies of $\mathcal{M}$.   

The new algorithm is based on the flow constraints that describe the policy space of the agent. Under Assumption \ref{assumption:unobservedmec}, the flow constraints given in \eqref{cons:positiveactions} - \eqref{cons:reach} disallow outflow from the observed maximal end components to the other states. We remove this assumption and allow outflow from UMECs. 

To find the minimum-information admissible policy we first create a modified MDP. The modified MDP has two copies of UMECs that are connected to each other with an action called $switch$. For a UMEC, while the original copy is connected to the other states, the duplicate copy is closed. We use the duplicate copies to represent the cases where the agent decides to stay in the UMEC.  

For MDP $\mathcal{M} = (S,\mathcal{A},\mathcal{P},s_0)$, we create a modified MDP $\bar{\mathcal{M}} = (\bar{S},\bar{\mathcal{A}},\bar{\mathcal{P}},s_0)$ as follows. Let $C_{end}$ be the set of states that belong to some UMEC of $\mathcal{M}$. For each $s \in C_{end}$, we create a duplicate state $\bar{s}$. Let $\bar{C}_{end}$ be the set of duplicate UMEC states. We define $\bar{S} := S \cup \bar{C}_{end}$. For all $s \in S $,  we define $\bar{\mathcal{A}}(s) := \mathcal{A}(s)$ and for all $a \in \bar{\mathcal{A}}(s)$, $q \in S$ we define $\bar{\mathcal{P}}_{s,a,q} := \mathcal{P}_{s,a,q}$. The duplicate state $\bar{s}$ has the action $a$ if and only if $a \in \mathcal{A}(s)$ and $\sum_{q \in C_{end}} \mathcal{P}_{s,a,q} = 1$. For every $\bar{s} \in \bar{C}_{end}$, $\bar{q} \in \bar{C}_{end}$, and $a \in \bar{\mathcal{A}}(\bar{s})$, we let $\bar{\mathcal{P}}_{\bar{s},a,\bar{q}} = \mathcal{P}_{s,a,q}$. For every state $s \in C_{end}$, we also add a new action $switch$ to $\bar{\mathcal{A}}(s)$ such that $\bar{\mathcal{P}}_{s,switch,\bar{s}} = 1$. 

Note that by definition $C_{reach}$ belongs to $C_{end}$. For the reachability constraint, we use the set of duplicate states $\bar{C}_{reach}$. For modified MDP $\bar{\mathcal{M}}$ and $\bar{C}_{end}$, we find the expected residence times of a minimum-information admissible policy with the following optimization problem

	\begin{subequations}
	\label{pr:mininfleakagemodified3}
	\begin{align}
	\inf \ &  \sum_{w \in W } x^{\pi}_{w} \iota^{\pi}_{w}  \label{eq:mininfrestime3}
	\\
	\normalfont \text{s. t. } \ &  x^{\pi}_{s} = \sum_{a \in \mathcal{A}(s)} x^{\pi}_{s,a}, & &  \forall s \in \bar{S} \setminus \bar{C}_{end} \label{cons:restimedefs3}
	\\ 
	&  x^{\pi}_{s,a} \geq 0, & &  \forall s \in \bar{S} \setminus \bar{C}_{end}, \	\forall a \in \bar{\mathcal{A}}(s) \label{cons:positiveactions3}
	\\
	& \sum_{a \in \bar{\mathcal{A}}(s)} x^{\pi}_{s,a} - \sum_{q \in S}  \sum_{a \in \bar{\mathcal{A}}(q)} x^{\pi}_{q,a}\mathcal{P}_{q,a,s} = \mathds{1}_{s_0}(s), & & \forall s \in \bar{S} \setminus \bar{C}_{end} \label{cons:floweqn3}
	\\
	& \sum_{q \in \bar{C}_{reach}} \ \sum_{s \in \bar{S} \setminus \bar{C}_{end}}  \sum_{a \in \bar{\mathcal{A}}(s)} x^{\pi}_{s,a}\mathcal{P}_{s,a,q}  + \mathds{1}_{s_0}(q) \geq \nu_{reach}, \label{cons:reach3}
	\end{align}
\end{subequations}
and synthesize the optimal policy $\bar{\pi}^*$.

\begin{remark}
The optimization problem given in $\eqref{pr:mininfleakagemodified3}$ does not include the policies that always stay in $C_{end}$, in the feasible set. However, we remark that it does not effect the optimality of the solution since the value of such a policy can also be achieved by a policy that enters and always stays in $\bar{C}_{end}$.
\end{remark}

 We describe the policy $\pi^*$ of the agent in the original MDP with Algorithm \ref{alg:polsyn3}. We use a memory element $switched$ that is $True$ if and only if $switch$ action is taken previously. We also synthesize a stationary policy $\pi_{stay}$ for MDP $\mathcal{M}$ that always stays in $C_{end}$. 

\begin{algorithm} [h]
	\caption{Synthesis of a minimum-information admissible policy } \label{alg:polsyn3}
	\begin{algorithmic}[1]
		\State \textbf{Input:} Current state $s$ in $\mathcal{M}$, $switched$,  $\bar{\pi}^*$, $\pi_{stay}$, and $C_{end}$.
		\State \textbf{Output:} The optimal policy $\pi^*$ of $\mathcal{M}$ and  $switched$.
		\If{$switched$}
			\State	$\pi := \pi_{stay}$.
		\ElsIf{$s \not\in C_{end}$}
			\State	$\pi^* := \bar{\pi}^*$.
		\Else
			\State $rnd := Unif[0,1]$.
			\If{$rnd \leq \bar{\pi}^*_{s,switch}$}
			\State $switched:= True$.
			\State $\pi^* = \pi_{stay}$.
			\Else
			\ForAll{$a \in \mathcal{A}(s)$} 
			\State $\pi^*_{s,a} = \dfrac{\bar{\pi}^*_{s,a}}{1-\bar{\pi}^*_{s,switch}}$.
			\EndFor
			\EndIf
		\EndIf
		\State \textbf{return} $\pi^*$ and $switched$
	\end{algorithmic}
\end{algorithm}

Note that the resulting policy is not stationary for the original MDP $\mathcal{M}$. The agent remembers whether it switched to the stay mode in the past. However, it is stationary for all states in $S \setminus C_{end}$. The inference problem is still meaningful since the policy does not change over time for the observed states.  

\begin{proposition} \label{proposition:stationaryisenough}
	For an MDP $\mathcal{M}$, the policy $\pi^*$ that is synthesized via the optimization problem given in \eqref{pr:mininfleakagemodified3} and Algorithm \ref{alg:polsyn3}, is a solution to the following problem
	\begin{align*}
	&\underset{ \pi \in \Pi^{Obs. St}(\mathcal{M})}{ \min}
	& &   \mathbb{E} [ \iota^{\pi}_{W,\xi} ] 
	\\
	& \text{s. t.}
	& & {\Pr} ^{\pi}(Reach[C_{reach}] ) \geq \nu_{reach}
	\end{align*}
	where $\xi$ is a random path generated under policy $\pi$.   
\end{proposition}

	\section{} \label{appendix:proofs}

		\begin{proof}[Proof of Proposition \ref{proposition:replace}]
		We first consider two cases:
		\begin{itemize}
			\item $\iota^{\pi}_w = \infty$ for a reachable state $w \in W$. 
			\item $\iota^{\pi}_w > 0$ for a state $w \in W$ and $w$ is recurrent under policy $\pi$.
		\end{itemize}
		
		Assume that the first case is possible. Since the path fragments of $\mathcal{M}^{\pi}$ that end with $w$ has a positive probability and $\iota^{\pi}_w = \infty$, the expected total information must be infinite. Thus, the first case is not possible. Assume that the second case is possible. Since the paths of $\mathcal{M}^{\pi}$ that visit $w$ infinitely often has a positive probability and $\iota^{\pi}_w > 0$, the expected total information must be infinite. The second case is also not possible. Hence, all observed states must be unreachable or must leak finite information and be transient.
		
		The expected total information of a transient or unreachable state $w \in W$ is
		\begin{subequations}
			\begin{align}
			\mathbb{E}[\iota^{\pi}_{w,\xi}] &= \sum_{n = 0}^{\infty} \Pr(N_{w, \xi} = n) n \iota^{\pi}_w
			\\
			&= \mathbb{E}[N_{w, \xi}] \iota^{\pi}_w
			\\
			&= x^{\pi}_w \iota^{\pi}_w
			\end{align}
		\end{subequations}
		where $N_{w, \xi}$ is the random variable that is the number of appearances of $w$ in $\xi$.
		
		The expected total information is 
		\begin{subequations}
			\begin{align}
			\mathbb{E}[\iota^{\pi}_{W,\xi}] &= \sum_{w \in W} \mathbb{E}[\iota^{\pi}_{w,\xi}]
			\\
			&= \sum_{w \in W} x^{\pi}_w \iota^{\pi}_w.
			\end{align}
		\end{subequations}
	\end{proof}
	
		\begin{proof}[Sketch of Proof for Proposition \ref{proposition:infmin}]
		If the optimal value of \eqref{pr:mininfleakagemodified} is infinite then any policy that satisfies the reachability constraints is the optimal policy. Otherwise, let $M$ be the optimal value of \eqref{pr:mininfleakagemodified}.
		$\iota^{\pi}_{s}$ given in \eqref{eq:restimeprob} is a lower semicontinuous function in the domain $x^{\pi}_{s,a} \geq 0$ where $a \in \mathcal{A}(s)$. The objective function of \eqref{pr:mininfleakagemodified} is a sum of lower semicontinuous functions and thus is a lower semicontinuous function in domain $x^{\pi}_{s,a} \geq 0$ for all $s \in S \setminus C_{end}$ and $a \in \mathcal{A}(s)$. For every $x^{\pi}_{w,a}$, that satisfies $x^{\pi}_{w} \iota^{\pi}_{w} \leq M$, is bounded where $w \in W$. Also every $x^{\pi}_{s,a}$ is bounded since a state $s\in S \setminus C_{end}$ must be transient. With the constraints \eqref{cons:positiveactions}-\eqref{cons:reach} the feasible region is a compact set. Since a lower semicontinuous function attains its infimum on a compact set, we conclude that the proposition holds.
	\end{proof}

	Before we proceed to the proof of Proposition \ref{proposition:cvx}, we give the following lemma that will be used in the proof.
\begin{lemma} \label{lemma:recisconvex}
	If $f:V \to \mathbb{R}$ is a positive, concave function where $V \subset \mathbb{R}^n$ is a convex set, then $\dfrac{1}{f(x)}$ is a convex function on $V$.
\end{lemma}
\begin{proof} [Proof of Lemma \ref{lemma:recisconvex}]
	Since $f$ and $\log$ are concave functions, $\log f$ is a concave function and consequently $-\log f$ is a convex function on $V$. Finally, $\exp(-\log f) = \dfrac{1}{ f}$ is a convex function on $V$, due to convexity of $\exp$ and $-\log f$ on $V$.
\end{proof}

\begin{proof} [Proof of Proposition \ref{proposition:cvx}]
	Let $f_1:Y_1 \to \mathbb{R} $ be a function such that $ f_1(p)= \sum_{i=1}^{n} p_i(1-p_i)$. Clearly $f_1$ is a positive, concave function on the convex domain $Y_1 =\{ p \ | \ p_1, \ldots , p_{n} \geq 0,\ \sum_{i=1}^n p_i = 1, \  \exists i,j \in [n], \ i \neq j , \ p_{i,j} >0 \}$. Let $f_2:Y_1 \to \mathbb{R}$ be a function such that $$f_2 := \dfrac{1}{f_1} = \dfrac{1}{\sum_{i=1}^{n} p_i(1-p_i)}.$$ By Lemma \ref{lemma:recisconvex}, $f_2$ is a convex function on the domain $Y_1$. A perspective function \cite{boyd2004convex} of $f_2$ is $g_1: Y_2 \to \mathbb{R}$ such that 
	
	\begin{align*}
	g_1 \left (x , \sum_{i = 1}^{n} x_i \right ) &= \sum_{i = 1}^{n} x_i f \left ( \dfrac{x}{\sum_{i = 1}^{n} x_i} \right )
	\\
	&= \sum_{i = 1}^{n} x_i f(p) 
	\end{align*}  
	where $$p_i = \dfrac{x_i}{\sum_{i=1}^{n}x_i}.$$ Due to the convexity property of perspective functions \cite{boyd2004convex}, $g_1$ is convex on $$Y_2= \left \lbrace (x , \sum_{i = 1}^{n} x_i) \ 
	| \ x_1, \ldots, x_n \geq 0 , \ \exists i,j \in [n], \  i \neq j, \ x_{i,j} >0  \right\rbrace$$ since $f_2$ is convex on $Y_1$. We eliminate the redundant dimension $\sum_{i=1}^{n} x_i$ and define $g_2: V_1 \to R$ such that $$g_2(x) = g_1 \left (x , \sum_{i = 1}^{n} x_i \right ).$$ $g_2$ is an affine transformation of $g_1$ and convex on  $$V_1= \left \lbrace x \ | \ x_1, \ldots, x_n \geq 0 , \ \exists i,j \in [n], \  i \neq j, \ x_{i,j} >0  \right\rbrace$$
	
	We introduce $V_0 = \left \lbrace x \ | \ x_1 =\ldots = x_n = 0 \right \rbrace$ and $V_{det} = \left \lbrace x \ | \ \exists i \in [n], \ \forall j \in [n], \ i \neq j, \  x_i > 0, \ x_j = 0  \right \rbrace$. Note that $V_0$, $V_1$, and $V_{det}$ are disjoint sets.
	
	Now we define $g:V \to R \cup \lbrace \infty \rbrace$ on $V=V_0 \cup V_1 \cup V_{det}$ such that $g(x) = 0$ if $x \in V_0$, $g(x) = g_1(x)$ if $x \in V_1 $, and $g(x) = \infty$ if $x \in V_{det}$.
	
	Clearly $g_3$ is convex on $V_0$ and $V_{det}$. We check all possible combinations for convexity where $\lambda \in [0,1]$: 
	
	$\bullet$  $\lambda g(v_1) + (1- \lambda) g(v_2) \geq g(\lambda v_1 + (1-\lambda)v_2)$ if $v_1 \in V_{det}$ and $v_2 \in V_0 \cup V_1$,
	
	$\bullet$  $\lambda g(v_1) + (1- \lambda) g(v_2) = g(\lambda Y_1 + (1-\lambda)v_2)$ if $v_1 \in V_0$ and $v_2 \in  V_1$.
	
	Hence $g$ is convex on $V$. 
	
	Now we represent the objective function of \eqref{pr:mininfleakagemodified} using $g$. Without loss of generality assume that the successor states of state $s$ are $q_1, \ldots, q_{|Succ(s)|}$ and the actions at state $s$ are $a_1, \ldots, a_{|\mathcal{A}(s)|}$. Note that $$x^{\pi}_{s} \iota^{\pi}_{s} = g(Px)$$ where $x = [x(s,a_1), \ldots, x(s,a_{|\mathcal{A}(s)|})]^T$ and $P$ is a $|Succ(s)| \times |\mathcal{A}(s)|$ matrix with $(i,j)$-th entry $\mathcal{P}_{s,a_j,q_i}$. 
	
	Since $x^{\pi}_{s} \iota^{\pi}_{s}$ is an affine mapping of $g$, $x^{\pi}_{s} \iota^{\pi}_{s}$ is convex on $T' = \{ x \in \mathbb{R}^{|\mathcal{A}(s)|} \ | \ \sum_{a\in \mathcal{A}(s)} x^{\pi}_{s,a} \mathcal{P}_{s,a,q} \geq 0 \}$ and consequently on $T= \{ x \in \mathbb{R}^{|\mathcal{A}(s)|} \ | \ x^{\pi}_{s,a} \geq 0 \} \subseteq T'$.
	
	The objective function \eqref{eq:mininfrestime} is a sum of convex functions and the constraints in \eqref{pr:mininfleakagemodified} are linear. Therefore, we conclude that the optimization problem is convex.
\end{proof}

	\begin{proof} [Proof of Proposition \ref{proposition:infoisalowerbound}]
	Due to the stochasticity of MDP, we might encounter the cases where the observer has no observation from a state and hence no sample for estimation. For such cases, denote $\sigma_{w, 0}$ for the MSE when there is no sample for estimation. Denote $\sigma_{w, +}$ for the MSE when there is at least one sample for estimation.

	  The MSE of the $q$-th element $(\sigma_w)_q$ is the estimation error for transition probability to the successor state $q$ such that $\sum_{q \in Succ(s)} (\sigma_w)_q = \sigma_w$.
	
	Denote the result of the successor state at time $t$ for the random path $\xi$ by $R_{t, \xi}$ where by definition $R_{-1, \xi} = s_0$ and $N_{w, \xi}$ for the number of times that state $w$ appears in $\xi$. We have
	\begin{subequations}
		\begin{align}
		(\sigma_w)_q  = & \Pr (N_{w, \xi} = 0 | \pi) (\sigma_{w, 0})_q +  \Pr (N_{w, \xi} > 0| \pi) (\sigma_{w, +})_q
		\\
		\geq & {\Pr}^{\pi}(Reach[w]) (\sigma_{w, +})_q 
		\\
		\geq & \dfrac{\Pr^{\pi}(Reach[w])}{I_{\xi | N_{w, \xi} > 0}(\pi_{w,q})}   \label{eq:cramerraoforpos}
		\\
		= &  \dfrac{\Pr^{\pi}(Reach[w])}{\sum_{t=0}^{\infty} I_{R_{\xi} | R_{t-1, \xi}, \ldots, R_{0, \xi} , N_{w, \xi} > 0} ( \mathcal{P}^{\pi}_{w,q} )}   \label{eq:chainruleofinformation}
		\\
		= &  \dfrac{\Pr^{\pi}(Reach[w])}{\sum_{t=0}^{\infty} I_{R_{t,\xi} | R_{t-1, \xi}, N_{w, \xi} > 0} ( \mathcal{P}^{\pi}_{w,q}  ) }    \label{eq:chainismarkovian}
		\\
		= &  \dfrac{ \Pr^{\pi}(Reach[w]) \mathcal{P}^{\pi}_{w,q}  (1- \mathcal{P}^{\pi}_{w,q}  )}{ \sum_{t=0}^{\infty} \Pr (R_{t-1} = w | N_{w, \xi} > 0)  }   
		\\
		= & \dfrac{ \Pr^{\pi}(Reach[w])^2 \mathcal{P}^{\pi}_{w,q}  (1- \mathcal{P}^{\pi}_{w,q}  ) } { \ x^{\pi}_{w} } 
		\end{align}
	\end{subequations}
	where \eqref{eq:cramerraoforpos} is due to Cram\'er-Rao bound, \eqref{eq:chainruleofinformation} is due to chain rule of the Fisher information \cite{zegers2015fisher}, and \eqref{eq:chainismarkovian} is due to Markovian property of paths.
	
	The MSE at state $w$ is bounded such that
	\begin{subequations}
		\begin{align}
		\sigma_w \geq & \sum_{q \in Succ(s)} \dfrac{  \Pr^{\pi}(Reach[w])^2  \mathcal{P}^{\pi}_{w,q}  (1- \mathcal{P}^{\pi}_{w,q} ) } { x^{\pi}_{w} } 
		\\
		=& \dfrac{ \Pr^{\pi}(Reach[w])^2  } { x^{\pi}_{w} \iota^{\pi}_{w}}.
		\end{align}
	\end{subequations}
\end{proof}

	\begin{proof} [Proof of Corollary  \ref{corollary:totalmsebound}]
	Total MSE at state $w$ is 
	\begin{subequations}
		\begin{align}
		\sigma_w \geq& \dfrac{ \Pr^{\pi}(Reach[w])^2 } { \ x^{\pi}_{w} \iota^{\pi}_{w}}.
		\end{align}
	\end{subequations}	
	The total MSE for the set of states $W$ is 
	\begin{subequations}
		\begin{align}
		\sum_{w \in W} \sigma_w \geq& \sum_{w \in W} \dfrac{  \Pr^{\pi}(Reach[w])^2 } { \ x^{\pi}_{w} \ \iota^{\pi}_{w} } \label{eq:actualcramerrao}
		\\
		\geq& \sum_{w \in W} \dfrac{\underset{w' \in W}{\min} \Pr^{\pi}(Reach[w]')^2 } { \ x^{\pi}_{w} \ \iota^{\pi}_{w} } 
		\\
		\geq& \dfrac{\underset{w' \in W}{\min} \Pr^{\pi}(Reach[w]')^2|W|^2} {\sum_{w \in W} \ x^{\pi}_{w} \ \iota^{\pi}_{w} } 
		\\
		= & \dfrac{\underset{w' \in W}{\min} \Pr^{\pi}(Reach[w'])^2 |W|^2 }{\mathbb{E}_{\xi} [ \iota^{\pi}_{W,\xi} ]}.
		\end{align}
	\end{subequations}	
\end{proof}

	\begin{proof}[Sketch of Proof for Proposition \ref{proposition:stationaryisenough}]
	The proof steps are as follows,
	\begin{itemize}
		\item show that a stationary policy $\bar{\pi}^*$ is optimal for modified MDP $\bar{\mathcal{M}}$ among all policies in $\Pi^{Obs. St}(\bar{\mathcal{M}})$,
		\item show that the minimum-information admissible of $\mathcal{M}$ is not lower than $\bar{\mathcal{M}}$,
		\item show that the expected total informations are equal for $\pi^*$ of $\mathcal{M}$ and $\bar{\pi}^*$ of $\bar{\mathcal{M}}$.
	\end{itemize}

	Consider the minimum-information admissible policy $\bar{\pi}^*$ for $\bar{\mathcal{M}}$. For every state $s \in \bar{S}$, we identify whether $\bar{\pi}^*$ makes $s$ recurrent or transient.
	
	Let $(C,D)$ be an original UMEC of $\bar{\mathcal{M}}$ and $\pi^1$ be a policy such that $\Pr^{\pi^1}(s_t \in C \text{ eventually always}) > 0 $. We claim that the expected total information under policy $\pi^1$ can also be achieved by a policy $\pi^2$ such that $\Pr^{\pi^2}(s_t \in C \text{ infinitely often}) = 0 $ since upon deciding to stay in $C$ the agent can first take $switch$ action and then take the same actions in the duplicate UMEC $\bar{C}$. Note that staying in $C$ or $\bar{C}$ does not affect the total information since both UMECs leak no information. Hence, we only look for policies that makes $C_{end}$ transient and $\bar{C}_{end}$ recurrent.

	Let $(C,D)$ be an end component of $\bar{\mathcal{M}}$ such that $C \cap C_{end}=\emptyset$ and there exists $w \in W$ and $w \in C$. We claim that a policy that always stays in $C$ visits an observed state infinitely often and leaks infinite information. If it does not, then there must exist a state $s \in C$ such that $s$ is recurrent and $s \not\in W$. Such a state $s$ must belong to a UMEC, but by construction it is not possible. Hence if there exists a policy that leaks finite information every state $s \in C$ must be transient. Also note that a state that does not belong to an end component must be transient by definition.

	We partition $\bar{S}$ into two sets: transient states $\bar{S} \setminus \bar{C}_{end}$ and recurrent states $\bar{C}_{end}$. Under any policy $\pi \in \Pi^{Obs.St}(\bar{\mathcal{M}})$ that makes the $\bar{C}_{end}$ recurrent and $\bar{S} \setminus \bar{C}_{end}$ transient, we have the flow equation 
	\begin{subequations}
		\begin{align*}
		& \sum_{a \in \mathcal{A}(s)} x^{\pi}_{s,a} - \sum_{q \in \bar{S} \setminus \bar{C}_{end} } \ \sum_{a \in A(q)} x^{\pi}_{q,a}\mathcal{P}_{q,a,s} = \mathds{1}_{s_0}(s),
	& & \forall s \in \bar{S} \setminus \bar{C}_{end}.
		\end{align*}
	\end{subequations}

Since we optimize over the observation stationary policies, the Proposition \ref{proposition:replace} still holds. The optimization problem given in \eqref{pr:mininfleakagemodified3} finds the state-action residence times of the optimal policy subject to the flow equation and the reachability constraint. The stationary policy synthesized via \eqref{alg:polsyn3} yields to the optimal expected residence times and hence is optimal.

	Let 	
	\begin{subequations}
		\label{pr:mininfleakage}
		\begin{align}
		&v^* = \underset{ \pi \in \Pi^{Obs. St}(\mathcal{M})}{\inf} \mathbb{E} [ \iota^{\pi}_{W,\xi} ] 
		\\
		& \text{s. t. }{\Pr} ^{\pi}(Reach[C_{reach}] ) \geq \nu_{reach}
	\end{align}
\end{subequations}
	and
		\begin{subequations}
		\begin{align}
		&\bar{v}^* =  \underset{ \bar{\pi} \in \Pi^{Obs. St}(\bar{\mathcal{M}})}{\inf} \mathbb{E} [ \iota^{\bar{\pi}}_{W,\xi}] 	
		\\
		& \text{s. t. }{\Pr} ^{\pi}(Reach[C_{reach}] ) \geq \nu_{reach}.
		\end{align}
\end{subequations}
Since every $\pi \in \Pi(\mathcal{M})$ is also realizable for $\bar{\mathcal{M}}$ with the same expected total information and reachability probabilities, we have $v^* \geq \bar{v}^*$.

Finally, we note that $\pi^*$ of $\mathcal{M}$ and $\bar{\pi}$ of $\bar{\mathcal{M}}$ yield to the same expected total information $\bar{v}^*$ since the expected residence times and the policies are the same at the observed states for both policies . Consequently, $\pi^*$ is a minimum-information admissible policy of $\mathcal{M}$.
\end{proof}

\end{document}